\documentclass[a4paper, 11pt, twoside]{scrartcl}
\pdfoutput=1

\title{Energy Stability of Explicit Runge-Kutta Methods for Non-autonomous or Nonlinear Problems}
\author{Hendrik Ranocha, David I. Ketcheson}
\date{June 4, 2020}

\usepackage[utf8]{luainputenc}
\usepackage[USenglish]{babel}
\usepackage{csquotes}

\usepackage[a4paper, top=0.51in, bottom=0.79in, left=0.999in, right=0.99in]{geometry}

\usepackage[plainpages=false,pdfpagelabels,hidelinks,unicode]{hyperref}
\makeatletter
\hypersetup{pdfauthor={\@author}}
\hypersetup{pdftitle={\@title}}
\makeatother

\usepackage[%
  backend=biber,
  style=numeric-comp,
  giveninits=true, uniquename=init, 
  natbib=true,
  url=true,
  doi=true,
  isbn=false,
  backref=false,
  maxnames=99,
  ]{biblatex}
\usepackage{amsmath}
\allowdisplaybreaks
\usepackage{amssymb}
\usepackage{commath}
\usepackage{mathtools}
\usepackage{bbm}
\usepackage{nicefrac}

\usepackage{siunitx}

\usepackage{amsthm}
\usepackage{thmtools}
\usepackage{etoolbox}
\makeatletter
\patchcmd{\thmt@setheadstyle}
 {\bgroup\thmt@space}
 {\thmt@space}
 {}{}
\patchcmd{\thmt@setheadstyle}
 {\egroup\fi}
 {\fi}
 {}{}
\makeatother
\declaretheoremstyle[
  bodyfont=\normalfont\itshape,
  headformat=\NAME\ \NUMBER\NOTE,
]{myplain}
\declaretheoremstyle[
  headformat=\NAME\ \NUMBER\NOTE,
]{mydefinition}
\newcommand{\envqed}{{\lower-0.3ex\hbox{$\triangleleft$}}}
\declaretheorem[style=myplain,numberwithin=section]{theorem}

\declaretheorem[style=myplain,numberlike=theorem]{proposition}
\declaretheorem[style=mydefinition,numberlike=theorem,qed=\envqed]{definition}
\declaretheorem[style=mydefinition,numberlike=theorem,qed=\envqed]{remark}
\declaretheorem[style=mydefinition,numberlike=theorem,qed=\envqed]{example}

\usepackage{ifluatex}
\ifluatex
  \usepackage[no-math]{fontspec}
\else
  \usepackage[T1]{fontenc}
\fi
\usepackage{newpxtext,newpxmath}

\usepackage{color}
\usepackage{graphicx}
\usepackage[small]{caption}
\usepackage{subcaption}

\usepackage{pgfplots}
\usepackage{tikzscale}
\pgfplotsset{compat=1.13}
\usepgfplotslibrary{groupplots}
\usepgfplotslibrary{external}
\begingroup\expandafter\expandafter\expandafter\endgroup
\expandafter\ifx\csname pdfsuppresswarningpagegroup\endcsname\relax
\else
\fi

\usepackage{booktabs}
\usepackage{rotating}
\usepackage{multirow}

\makeatletter
\newcommand\mynewtag[2]{#1\def\@currentlabel{#1}\label{#2}}
\makeatother

\usepackage{multicol}
\usepackage{enumitem}

\usepackage{calc}
\usepackage{xparse}

\usepackage{forest}
\forestset{
  */.style={
    delay+={append={[]},}
  },
  rooted tree/.style={
    for tree={
      grow'=90,
      parent anchor=center,
      child anchor=center,
      s sep=2.5pt,
      if level=0{
        baseline
      }{},
      delay={
        if content={*}{
          content=,
          append={[]}
        }{}
      }
    },
    before typesetting nodes={
      for tree={
        circle,
        fill,
        minimum width=3pt,
        inner sep=0pt,
        child anchor=center,
      },
    },
    before computing xy={
      for tree={
        l=5pt,
      }
    }
  }
}
\DeclareDocumentCommand\rootedtree{o}{\Forest{rooted tree [#1]}}

\NewDocumentCommand{\mat}{mo}{%
  \IfValueTF{#2}{%
    \underline{\underline{#1}}{#2}
  }{%
    \underline{\underline{#1}}\,
  }%
}
\newcommand{\diag}{\operatorname{diag}}

\AtBeginDocument{%
  
}

\newcommand{\scp}[2]{\left\langle{#1,\, #2}\right\rangle}
\newcommand{\scps}[2]{\langle{#1,\, #2}\rangle}

\DeclarePairedDelimiterX\newset[1]\lbrace\rbrace{\setaux #1||\endsetaux}
\def\setaux#1|#2|#3\endsetaux{\if\relax\detokenize{#2}\relax #1 \else #1 \;\delimsize\vert\; #2 \fi}
\renewcommand{\set}[1]{\newset*{#1}}

\newcommand{\vect}[1]{\begin{pmatrix} #1 \end{pmatrix}}

\newcommand{\I}{\operatorname{I}}

\newcommand{\imax}{{i_\mathrm{max}}}

\AtBeginDocument{%
  \let\rho\varrho
  \let\phi\varphi
  \let\epsilon\varepsilon
}
\newcommand{\N}{\mathbb{N}}
\newcommand{\Z}{\mathbb{Z}}
\newcommand{\R}{\mathbb{R}}

\newcommand{\dt}{\Delta t}
\renewcommand{\H}{\mathcal{H}}

\renewcommand{\O}{\mathcal{O}}
\renewcommand{\i}{\mathrm{i}}

\newcommand{\normLip}[1]{\abs{#1}_{\mathrm{Lip}}}

\usepackage{xspace}

\newcommand{\eg}[0]{{e.g.\@}\xspace}

\newcommand{\shortlong}[2]{#2}
\newcommand{\sep}{, }

\begin{document}

\maketitle

\begin{abstract}
  Many important initial value problems have the property that energy
is non-increasing in time. Energy stable methods, also referred
to as strongly stable methods, guarantee the same property discretely.
We investigate requirements for conditional energy stability of explicit
Runge-Kutta methods for nonlinear or non-autonomous problems.
We provide both necessary and sufficient conditions for energy stability
over these classes of problems.
Examples of conditionally energy stable schemes are constructed and an
example is given in which unconditional energy stability is obtained with an
explicit scheme.


\end{abstract}

\section{Introduction}
\label{sec:introduction}

Ever since the construction of numerical methods for ordinary and
(time-dependent) partial differential equations (ODEs and PDEs, respectively),
their stability has been an important and active topic of research.
Monotonicity, meaning that the norm of the solution is bounded by its
initial value, is a particularly exacting stability property.
For equations, such as parabolic PDEs, that contain a significant amount of dissipation,
any reasonable numerical method will typically preserve monotonicity
under an appropriate time step restriction.
In contrast, for non-dissipative problems such as hyperbolic PDEs
(and their slightly dissipative semidiscretizations),
common time discretizations may not preserve monotonicity under any finite step size.

For nonlinear problems, such stability properties of numerical methods are
interesting and relevant in practice, \eg for computational fluid dynamics
\cite{rojas2019robustness}. But important applications are not restricted
to such nonlinear problems. Linear problems with time-dependent operators
occur for example if the passive transport of a tracer in a fluid is
considered. Other examples occur in the field of plasma physics when a
hybrid method is used and the motion of macro-particles is coupled in
an operator splitting approach to the evolution of the magnetic field,
which is determind as the solution of a linear hyperbolic PDE with
time-dependent coefficients \cite{koenders2015dynamical, ranocha2018numerical}.
Such problems are not only much more complex in terms of stability (as outlined
in this article) but also in terms of the asymptotic error growth
\cite{offner2019error}.

The energy method is an effective tool to get stability estimates, e.g. for
hyperbolic PDEs \cite{kreiss2004initial, gustafsson2008high}. Using summation
by parts operators \cite{svard2014review, fernandez2014review}, these can be
transferred efficiently to the semidiscrete level for many different kinds
of schemes \cite{kreiss1974finite, nordstrom2001finite, gassner2013skew,
ranocha2016summation}. However, applying the same approach in time yields
implicit methods \cite{nordstrom2013summation, boom2015high, ranocha2019some,
ranocha2020class, friedrich2019entropy}.  Classical nonlinearly stable methods,
such as algebraically stable Runge-Kutta methods, are also implicit.
For hyperbolic problems, such implicit methods are
usually less efficient than explicit ones.  It is possible to obtain
conditional energy stability with explicit methods by using modifications
that go outside the class of Runge-Kutta methods;
e.g. projection methods
\cite{hairer2006geometric, calvo2006preservation, calvo2010projection}
and relaxation Runge-Kutta schemes \cite{ketcheson2019relaxation,
ranocha2020relaxation, ranocha2020general, ranocha2020fully}.
Another possibility, studied in particular in the context of hyperbolic PDEs,
is to add artificial dissipation, spectral viscosity, or filtering
\cite{offner2018artificial, offner2020analysis, sun2019enforcing}. If these
modifications are implemented after a full Ruge-Kutta step, they belong to
the same general class of projection methods as relaxation and orthogonal
projection schemes.

Before trying to modify existing algorithms to prove stability results, it is
of course interesting to know whether these modifications are really necessary
to get stability or only for a proof thereof. Additionally, modifications of
explicit methods can make the algorithm more complicated and some versions
can even destroy other desired properties while trying to impose stability
\cite{offner2020analysis, ranocha2020relaxationHamiltonian}. Hence,
it is interesting to know what can be achieved within the class
of explicit Runge-Kutta methods without modifications. In this setting, results
have been obtained for problems that include a certain amount of dissipation
\cite{dahlquist1979generalized, higueras2005monotonicity}.
Recently this topic has again attracted the interest of researchers and
several results (using the term {\em strong stability}) for linear,
time-independent operators have been discovered
\cite{ranocha2018L2stability, sun2017stability, sun2019strong}. Nonlinear
problems have been investigated in \cite{ranocha2020strong}, where many
non-existence results for energy stable and strong stability preserving
(SSP) methods of order two and greater have been proved.
A first order accurate energy stable SSP method for autonomous problems
has also been discovered therein.

This article extends these previous works considerably by studying both
time-dependent linear and autonomous nonlinear problems. After introducing
the notation and reviewing some basic results in Section~\ref{sec:rk-methods},
the focus lies on time-dependent linear operators in Section~\ref{sec:linear}.
The main result, Theorem~\ref{thm:Lipschitz-impossible}, gives necessary
conditions for conditional energy stability in this
setting.  These conditions are not satisfied by any known
Runge-Kutta scheme we are aware of.  However, an example of
a scheme fulfilling these necessary conditions is given, and is proved to be
energy stable for a restricted class of relevant problems
(Proposition~\ref{pro:probbaly-stable-scheme-1}).

Next, autonomous nonlinear problems are studied in Section~\ref{sec:autonomous}.
The necessary conditions in this setting are, perhaps surprisingly, weaker than in the
non-autonomous linear case. These conditions are based on an expansion of
the change of energy \eqref{eq:estimate-RK-trees}, which is also used to
study sufficient conditions for energy stability. Based thereon,
we give a procedure for developing energy stable schemes,
and give examples of schemes of second and third order
(Theorem~\ref{thm:probably-stable-scheme} and
Example~\ref{ex:stable-rk_3_5-Mathematica}).

While most of the paper is devoted to the guarantee of stability over a
whole class of semibounded problems, in Section~\ref{sec:special-example} we ask
whether an explicit Runge-Kutta method can be unconditionally energy stable
for some specific problem.  We show that this is impossible if the problem is
linear, but -- surprisingly -- we give an example of unconditional stability
for a nonlinear problem.
Finally, in Section~\ref{sec:summary}, the results are summed up, open
questions are discussed, and directions of future research are outlined.

\section{Energy Evolution by Runge-Kutta Methods}
\label{sec:rk-methods}

Consider a time-dependent initial value problem
\begin{equation}
\label{eq:ode}
\begin{aligned}
  \od{}{t} u(t) &= f(t, u(t)),
  && t \in (0,T),
  \\
  u(0) &= u_0,
\end{aligned}
\end{equation}
in a real Hilbert space $\H$ with inner product $\scp{\cdot}{\cdot}$, inducing
the norm $\norm{\cdot}$.  We refer to $\norm{\cdot}^2$ as the energy.

\subsection{Energy Stability}

For a smooth solution of \eqref{eq:ode}, the time derivative of the energy is
\begin{equation}
  \od{}{t} \norm{u(t)}^2
  =
  2 \scp{u(t)}{\od{}{t} u(t)}
  =
  2 \scp{u(t)}{f(t, u(t))}.
\end{equation}

\begin{definition}
  The right hand side $f\colon [0,T] \times \H \to \H$ of \eqref{eq:ode} is
  \emph{semibounded}, if
  \begin{equation}
    \forall u \in \H, t \in [0,T]\colon
    \quad
    \scp{u}{f(t, u)} \leq 0.
  \end{equation}
  If $f$ is semibounded, the ODE \eqref{eq:ode} will sometimes also be called semibounded.
\end{definition}
\begin{remark}
  The results in this work extend to complex Hilbert spaces if one assumes that
  the real part of the inner product $\scp{u}{f(t, u)}$ is non-positive.
\end{remark}

Thus, the energy of any smooth solution of \eqref{eq:ode} is bounded by
its initial value if $f$ is semibounded. However, an approximate solution obtained
by a numerical method does not necessarily satisfy this inequality. For example,
applying one step of the explicit Euler method to \eqref{eq:ode} yields the new
value $u_+ = u_0 + \dt f(0, u_0)$, satisfying
\begin{equation}
  \norm{u_+}^2
  =
  \norm{u_0 + \dt f(0, u_0)}^2
  =
  \norm{u_0}^2
  + \underbrace{2 \dt \scp{u_0}{f(0, u_0)}}_{\leq 0}
  + \underbrace{\dt^2 \norm{f(0, u_0)}^2}_{\geq 0}.
\end{equation}
Thus, for a general semibounded $f$, the norm of the numerical solution can increase
during one time step, e.g. if $\scp{u_0}{f(0, u_0)} = 0$. In particular, this happens
if $f(t, u) = L(t) u$, where $L(t) \neq 0$ is a skew-symmetric operator.

\begin{definition}
  A one-step numerical scheme for approximating the solution of \eqref{eq:ode} is
  \emph{conditionally energy stable} with respect to a class
  ${\mathcal F}$ of semibounded problems if for each $f\in {\mathcal F}$ there exists
  $\dt_{\mathrm{max}}>0$ such that $\norm{u_+}^2 \leq \norm{u_0}^2$ for all $0 < \dt \le \dt_{\mathrm{max}}$.
\end{definition}
Here $\dt_{\mathrm{max}}$ may depend on $f$, $u_0$, and the method itself.
In the following we will consider the classes ${\mathcal F}$ of linear non-autonomous
and nonlinear autonomous problems.  We will often omit the word \emph{conditional}
for brevity.

\subsection{Runge-Kutta Methods}

A general (explicit or implicit) Runge-Kutta method with $s$ stages can be
described by its Butcher tableau \cite{hairer2008solving, butcher2016numerical}
\begin{equation}
\label{eq:butcher}
\begin{array}{c | c}
  c & A
  \\ \hline
    & b
\end{array}
\end{equation}
where $A \in \R^{s \times s}$ and $b, c \in \R^s$. For \eqref{eq:ode}, a step
from $u_0$ to $u_+$ is given by
\begin{equation}
\label{eq:RK-step}
  u_i
  =
  u_0 + \dt \sum_{j=1}^{s} a_{ij} \, f(c_j \dt, u_j),
  \qquad
  u_+
  =
  u_0 + \dt \sum_{i=1}^{s} b_{i} \, f(c_i \dt, u_i).
\end{equation}
Here, $u_i$ are the stage values of the Runge-Kutta method. It is also
possible to express the method via the slopes $k_i = f(c_i \dt, u_i)$.

Using the stage values $u_i$ as in \eqref{eq:RK-step},
the change in energy can be written
after some simplifications using the symmetry of the inner product as
\cite[equation (357e)]{butcher2016numerical}
\begin{multline}
\label{eq:estimate-RK}
  \norm{ u_+ }^2 - \norm{ u_0 }^2
  =
  2 \dt \sum_{i=1}^{s} b_{i} \scp{ u_i }{ f(c_i \dt, u_i) }
  \\
  + (\dt)^2 \left[
    \sum_{i,j=1}^{s} \left( b_i b_j - b_{i} \, a_{ij} - b_j a_{ji} \right)
      \scp{ f(c_i \dt, u_i) }{ f(c_j \dt, u_j) }
  \right].
\end{multline}
The first term on the right hand side is consistent with
$\int_{t_0}^{t_0 + \dt} 2 \scp{u(t)}{f(t, u(t))} \dif t$,
if the Runge-Kutta method is consistent, i.e. $\sum_{i=1}^{s} b_i = 1$.
Semiboundedness of $f$ implies that this term is non-positive if all $b_i$
are non-negative.

We recall the classical property of algebraic stability, which guarantees
energy stability for semibounded operators; cf.
\cite[section~357]{butcher2016numerical} and references cited therein.
\begin{definition}
  A Runge-Kutta method is \emph{algebraically stable} if $b_i \geq 0$ for all $i$
  and the matrix with entries $(b_i b_j - b_{i} \, a_{ij} - b_j a_{ji})_{i,j}$
  is negative semidefinite.
\end{definition}
Comparison with \eqref{eq:estimate-RK} shows that algebraically stable
Runge-Kutta methods are energy stable
for any time step size $\dt > 0$, cf. \cite[section~357]{butcher2016numerical}
and references cited therein. While there are Runge-Kutta methods with these nice
stability properties such as Gauß, Radau IA/IIA or Lobatto IIIC schemes, these
are necessarily implicit.

For linear and semibounded problems \eqref{eq:ode} with constant coefficients,
several results concerning the conditional energy stability of explicit Runge-Kutta methods
have been achieved \cite{tadmor2002semidiscrete, ranocha2018L2stability,
sun2017stability, sun2019strong} (note that the term \emph{conditional energy stability} herein
is precisely what is meant by \emph{strong stability} in the latter works).
Typically, conditional energy stability can be guaranteed
for problems $f(t, u) = L u$ in this class under a time step restriction of the
form $\dt \leq C \norm{L}^{-1}$, corresponding to a classical CFL criterion for
discretizations of hyperbolic conservation laws \cite{tadmor2002semidiscrete,
ranocha2018L2stability, sun2017stability, sun2019strong}.
Similar results have been obtained for some first order accurate schemes and
autonomous semibounded nonlinear problems \cite{ranocha2020strong}. In the latter setting,
the maximal time step is proportional to the inverse of the Lipschitz constant
of the nonlinear right-hand side $f$ of the ODE.

\section{Non-autonomous Linear Operators}
\label{sec:linear}

In this section, the special case of non-autonomous linear operators is studied.
Hence,
\begin{equation}
\label{eq:ode-linear}
\begin{aligned}
  \od{}{t} u(t) &= L(t) u(t),
  && t \in [0,T],
  \\
  u(0) &= u_0,
\end{aligned}
\end{equation}
is considered as special case of \eqref{eq:ode}.

To formulate our results, we need the following definitions regarding the
abscissae, or nodes, of a Runge-Kutta method.

\begin{definition}
\label{def:RK-nodes}
  \phantom{space}
  \begin{itemize}
    \item
    We say node $c_k$ of a Runge-Kutta method is \emph{distinct} if there is no
    other node $c_j=c_k$ such that $j\ne k$.

    \item
    A Runge-Kutta method is said to be \emph{non-confluent} if each of its
    nodes is distinct. Otherwise, it is called \emph{confluent}.

    \item
    We say the node $c_k$ of a Runge-Kutta method is a \emph{quadrature node}
    if $b_k\ne 0$.
  \end{itemize}
\end{definition}

\subsection{Main Result}

For linear problems with constant coefficients, some common and practical
Runge-Kutta methods have been shown to be conditionally energy stable
\cite{tadmor2002semidiscrete, ranocha2018L2stability, sun2017stability, sun2019strong}.
For linear problems with varying coefficients, energy stability is more difficult to attain.
Given almost any Runge-Kutta method, we can choose a non-autonomous problem \eqref{eq:ode-linear}
that makes the given method behave like Euler's method, leading to energy growth.
\begin{theorem}
\label{thm:Lipschitz-impossible}
  An explicit Runge-Kutta method with a distinct quadrature node cannot be energy
  stable for semibounded linear problems \eqref{eq:ode-linear} under a time
  step restriction depending only on an upper bound of $\norm{L(t)}$ and the
  Lipschitz constant of $t \mapsto L(t)$.
\end{theorem}
\begin{proof}
  Given $\dt > 0$, choose $u_0 \neq 0$,
  $L(t) = \lambda(t) \begin{psmallmatrix} 0 & -1 \\ 1 & 0 \end{psmallmatrix}$
  with $\lambda(c_k \dt) > 0$ sufficiently small and $\lambda(c_j) = 0$ for all
  $j \neq k$.
  By Kirszbraun's theorem \cite[Theorem~1.31]{schwartz1969nonlinear}, the
  function $\lambda$ can be continued as a Lipschitz continuous function with
  arbitrarily small Lipschitz constant depending on $\lambda(c_k \dt) > 0$.
  Hence, $\norm{L(t)}$ and the Lipschitz constant of $t \mapsto L(t)$ can be made
  arbitrarily small.

  Then, the first step of the given explicit Runge-Kutta method yields
  \begin{equation}
  \begin{aligned}
    u_+
    &=
    u_0 + \dt \sum_{i=1}^s b_i L(c_i \dt) u_i
    =
    u_0 + \dt b_k L(c_k \dt) u_k
    \\
    &=
    u_0 + \dt b_k L(c_k \dt) \left( u_0 + \dt \sum_{j=1}^{k-1} a_{kj} L(c_j \dt) u_j \right)
    =
    u_0 + \dt b_k L(c_k \dt) u_0.
  \end{aligned}
  \end{equation}
  This is equivalent to an explicit Euler step with time step $b_k \dt \neq 0$.
  Consider the matrix $L(c_k \dt)$.  Since this matrix (which is not identically
  zero) is skew-symmetric and injective,
  \begin{equation}
    \norm{u_+}^2 - \norm{u_0}^2
    =
    b_k^2 \dt^2 \norm{L(c_k \dt) u_0}^2
    > 0.
  \end{equation}
  Hence, the explicit Runge-Kutta method is not energy stable.
\end{proof}

The function $L$ appearing in the proof can be made arbitrarily smooth by
considering classical cut-off functions (Friedrichs mollifier).
The proof holds also for linear scalar problems in the complex plane if
$\begin{psmallmatrix} 0 & -1 \\ 1 & 0 \end{psmallmatrix}$ is replaced by
the imaginary unit $\i$.

\begin{remark}
\label{rem:methods-currently-used}
  In a non-confluent method, every node is distinct, so
  Theorem~\ref{thm:Lipschitz-impossible} implies that such methods cannot be energy
  stable.
  Moreover, it seems that all Runge-Kutta methods currently used in practice
  have at least one distinct quadrature node. The only schemes not covered by
  the theorem are those for which \emph{every quadrature node is repeated}.
  As far as we know, the only time integration schemes used in practice
  that can be reformulated as Runge-Kutta methods and use only repeated
  quadrature nodes are certain deferred correction methods. Whether these are energy
  stable or not requires further investigation.
\end{remark}

\begin{remark}
  The proof shows additionally that $a_{kk} \neq 0$ (taking $k$ as in the proof)
  is a necessary condition for an implicit Runge-Kutta scheme to be energy stable.
  In particular, the Lobatto~IIIA and IIIB schemes cannot be energy stable because
  they have a zero row or column in $A$ and only positive $b_i$.
\end{remark}

\begin{remark}
  The technique used in the proof of Theorem~\ref{thm:Lipschitz-impossible} cannot
  be extended to arbitrary confluent Runge-Kutta schemes. Indeed, consider the
  following Runge-Kutta method with nodes $c = (0, 1, 0)^T$
  for \eqref{eq:ode-linear}
  \begin{equation}
  \label{eq:counterexample-method-thm:Lipschitz-impossible}
  \begin{aligned}
    u_1 &= u_0,
    \\
    u_2 &= u_0 + \dt L(0) u_1,
    \\
    u_3 &= u_0 + \dt L(0) u_1 - \dt L(\dt) u_2,
    \\
    u_+ &= u_0 + \frac{\dt}{2} L(0) (- u_1 + 3 u_3).
  \end{aligned}
  \end{equation}
  If $L(t)$ is skew-symmetric, $L(0) \neq 0$, and $L(\dt) = 0$,
  \begin{equation}
    u_+
    =
    u_0 + \frac{\dt}{2} L(0) (- u_0 + 3 u_0 + 3 \dt L(0) u_0)
    =
    u_0 + \dt L(0) u_0 + \frac{3}{2} \dt^2 L(0)^2 u_0,
  \end{equation}
  and
  \begin{equation}
  \begin{aligned}
    \norm{u_+}^2 - \norm{u_0}^2
    \shortlong{}{&=
    2 \dt \scp{u_0}{L(0) u_0}
    + 3 \dt^2 \scp{u_0}{L(0)^2 u_0}
    + \dt^2 \norm{L(0) u_0}^2
    \\&\quad
    + 3 \dt^3 \scp{L(0) u_0}{L(0)^2 u_0}
    + \frac{9}{4} \dt^4 \norm{L(0)^2 u_0}^2
    \\}
    &=
    - 2 \dt^2 \norm{L(0) u_0}^2 + \frac{9}{4} \dt^4 \norm{L(0)^2 u_0}^2.
  \end{aligned}
  \end{equation}
  If $\dt$ is sufficiently small, $\norm{u_+}^2 - \norm{u_0}^2 < 0$.

  Similarly, if $L(t)$ is skew-symmetric, $L(0) = 0$, and $L(\dt) \neq 0$,
  $u_+ = u_0$ and the first time step is energy stable.

  However, this does not mean that the scheme
  \eqref{eq:counterexample-method-thm:Lipschitz-impossible}
  is energy stable for semibounded operators. Indeed, its stability function
  \begin{equation}
    \phi(z)
    =
    \frac{\det(\I - z A + z e b^T)}{\det(\I - z A)}
    =
    1 + z - \frac{3}{2} z^3
  \end{equation}
  satisfies $\abs{\phi(\i y)} > 1$ for $0 \neq y \in \R$. Hence, the scheme is not
  even stable for linear skew-symmetric operators with constant coefficients.
\end{remark}

\begin{remark}
  The proof of Theorem~\ref{thm:Lipschitz-impossible} above relies on the fact
  that if a Runge-Kutta method has a distinct node $c_k$,
  then the value of the corresponding stage derivative $L(c_k \dt)$ can be
  chosen independently of the values of all other stages. For methods
  with only duplicated nodes, the values of the stage derivatives become
  coupled and the problematic construction in the proof is precluded.
\end{remark}

The assumption of a distinct quadrature node used in
Theorem~\ref{thm:Lipschitz-impossible} may be necessary.
This conjecture is supported by the following result.
The proof of Theorem~\ref{thm:Lipschitz-impossible} relies on the
construction of a suitable operator $L$ such that
$L(t_1) L(t_2) = L(t_2) L(t_1)$ for all $t_1, t_2 \in [0,T]$ and
$\forall t \in [0,T]\colon L(t) = - L(t)^T$. For such operators, the following
confluent Runge-Kutta method is energy stable.
\begin{proposition}
\label{pro:probbaly-stable-scheme-1}
  The Runge-Kutta scheme with coefficients
  \begin{equation}
  \label{eq:probably-stable-scheme}
  \begin{aligned}
    \begin{array}{c|cccc}
    0 & & & &  \\
    1 & 1 & & & \\
    0 & 1 & -1 & & \\
    1 & -1 & 1 & 1 & \\
    \hline
    & \nicefrac{1}{4} & \nicefrac{1}{4} & \nicefrac{1}{4} & \nicefrac{1}{4} \\
    \end{array}
  \end{aligned}
  \end{equation}
  is conditionally energy stable for the class of linear non-autonomous ODEs
  \eqref{eq:ode-linear}
  where the operator $L$ is bounded and satisfies
  $\forall t_1, t_2 \in [0,T]\colon L(t_1) L(t_2) = L(t_2) L(t_1)$ and
  $\forall t \in [0,T]\colon L(t) = - L(t)^T$.
\end{proposition}
\begin{proof}
  Using $L_0 := L(0)$ and $L_1 := L(\dt)$, the change of the norm can be
  calculated explicitly for general $L$ as
  \begin{align*}
  \stepcounter{equation}\tag{\theequation}
  \label{eq:probably-stable-scheme-proof1}
    &
    \norm{u_+}^2 - \norm{u_0}^2
    =
    \dt \biggl(
        \scps{ u_0 }{ L_0 u_0 }
      + \scps{ u_0 }{ L_1 u_0 }
    \biggr)
    \\
    &\quad
    + \frac{1}{2} \dt^2 \biggl(
        \scps{ u_0 }{ L_0^2 u_0 }
      - \scps{ u_0 }{ L_0 L_1 u_0 }
      + \scps{ u_0 }{ L_1 L_0 u_0 }
      + \scps{ u_0 }{ L_1^2 u_0 }
      \\
      &\qquad\quad
      + \frac{1}{2} \norm{ L_0 u_0 }^2
      + \scps{ L_0 u_0 }{ L_1 u_0}
      + \frac{1}{2} \norm{ L_1 u_0 }^2
    \biggr)
    \\
    &\quad
    + \frac{1}{2} \dt^3 \biggl(
      - \scps{ u_0 }{ L_0 L_1 L_0 u_0 }
      + \scps{ u_0 }{ L_1 L_0^2 u_0 }
      - \scps{ u_0 }{ L_1 L_0 L_1 u_0 }
      + \scps{ u_0 }{ L_1^2 L_0 u_0 }
      \\
      &\qquad\quad
      + \frac{1}{2} \scps{ L_0 u_0 }{ L_0^2 u_0 }
      - \frac{1}{2} \scps{ L_0 u_0 }{ L_0 L_1 u_0 }
      + \frac{1}{2} \scps{ L_0 u_0 }{ L_1 L_0 u_0 }
      + \frac{1}{2} \scps{ L_0 u_0 }{ L_1^2 u_0 }
      \\
      &\qquad\quad
      + \frac{1}{2} \scps{ L_0^2 u_0 }{ L_1 u_0 }
      + \frac{1}{2} \scps{ L_1 u_0 }{ L_1 L_0 u_0 }
      - \frac{1}{2} \scps{ L_1 u_0 }{ L_0 L_1 u_0 }
      + \frac{1}{2} \scps{ L_1 u_0 }{ L_1^2 u_0 }
    \biggr)
    \\
    &\quad
    + \frac{1}{2} \dt^4 \biggl(
      - \scps{ u_0 }{ L_1 L_0 L_1 L_0 u_0 }
      - \frac{1}{2} \scps{ L_0 u_0 }{ L_0 L_1 L_0 u_0 }
      + \frac{1}{2} \scps{ L_0 u_0 }{ L_1 L_0^2 u_0 }
      \\
      &\qquad\quad
      - \frac{1}{2} \scps{ L_0 u_0 }{ L_1 L_0 L_1 u_0 }
      + \frac{1}{2} \scps{ L_0 u_0 }{ L_1^2 L_0 u_0 }
      + \frac{1}{8} \norm{L_0^2 u_0 }^2
      - \frac{1}{4} \scps{ L_0^2 u_0 }{ L_0 L_1 u_0 }
      \\
      &\qquad\quad
      + \frac{1}{4} \scps{ L_0^2 u_0 }{ L_1 L_0 u_0 }
      + \frac{1}{4} \scps{ L_0^2 u_0 }{ L_1^2 u_0 }
      + \frac{1}{2} \scps{ L_1 u_0 }{ L_1 L_0^2 u_0 }
      \\
      &\qquad\quad
      - \frac{1}{4} \scps{ L_0 L_1 u_0 }{ L_1 L_0 u_0 }
      + \frac{1}{8} \norm{ L_1 L_0 u_0 }^2
      + \frac{1}{4} \scps{ L_1 L_0 u_0 }{ L_1^2 u_0 }
      \\
      &\qquad\quad
      - \frac{1}{2} \scps{ L_1 u_0 }{ L_0 L_1 L_0 u_0 }
      + \frac{1}{2} \scps{ L_1 u_0 }{ L_1^2 L_0 u_0 }
      - \frac{1}{2} \scps{ L_1 u_0 }{ L_1 L_0 L_1 u_0 }
      \\
      &\qquad\quad
      + \frac{1}{8} \norm{ L_0 L_1 u_0 }^2
      - \frac{1}{4} \scps{ L_0 L_1 u_0 }{ L_1^2 u_0 }
      + \frac{1}{8} \norm{ L_1^2 u_0 }^2
    \biggr)
    + \O( \dt^5 ).
  \end{align*}
  Inserting the assumptions on $L$ (skew-symmetry of $L_{0,1}$ and
  $L_0 L_1 = L_1 L_0$), this equation reduces to (see Appendix~\ref{app:details})
  \begin{align*}
  \stepcounter{equation}\tag{\theequation}
  \label{eq:probably-stable-scheme-proof2}
    \norm{u_+}^2 - \norm{u_0}^2
    &=
    - \frac{1}{4} \norm{L_0 u_0 - L_1 u_0}^2 \dt^2 
    \\
    &\quad
    + \frac{1}{16} \biggl(
          \norm{L_0^2 u_0 }^2
      - 6 \norm{ L_0 L_1 u_0 }^2
      +   \norm{ L_1^2 u_0 }^2
    \biggr) \dt^4 
    + \O( \dt^5 ).
  \end{align*}
  The coefficient multiplying $\dt^2$ is nonpositive.
  If this coefficient is negative, the energy is non-increasing if the
  time step $\dt$ is small enough.

  Otherwise, $L_0 u_0$ must be equal to $L_1 u_0$ and the coefficient multiplying $\dt^4$
  is $- \frac{1}{4} \norm{ L_0^2 u_0 }^2 \leq 0$.
  If this coefficient vanishes, all products of higher powers of $L_0, L_1$ and $u_0$
  must vanish, too, since $L_0^m L_1^n u_0 = L_0^{m+n} u_0$ for arbitrary
  $m, n \in \mathbb{N}$ and $L_0^{m+n} u_0 = 0$ for $m+n \ge 2$.
  Consequently, $\norm{u_+}^2 = \norm{u_0}^2$.
\end{proof}

\subsection{Extensions}

Using an additional technical assumption, the construction used to prove
Theorem~\ref{thm:Lipschitz-impossible} can be used to show that the
growth of the norm is unbounded.
\begin{theorem}
\label{thm:Lipschitz-growth}
  Let an explicit Runge-Kutta method be given. Suppose there exists a
  distinct quadrature node $c_k$ such that for all $j \ne k$
  the difference $c_j - c_k$ is not an integer.
  Then there exists a semibounded problem \eqref{eq:ode-linear} such that the
  numerical solution given by the method grows monotonically without bound.
\end{theorem}
\begin{proof}
  The construction used in the proof of Theorem~\ref{thm:Lipschitz-impossible}
  can be applied to all steps of the given Runge-Kutta method simultaneously,
  since $\forall l \in \Z, j \neq k\colon c_k \neq c_j + l$. Hence,
  for every step from $u^{(n)}$ to $u^{(n+1)}$,
  \begin{equation}
    \norm{u^{(n+1)}}^2 - \norm{u^{(n)}}^2
    =
    b_k^2 \Delta t^2 \abs{\lambda(c_k \dt)}^2 \norm{u^{(n)}}^2
    >
    0.
  \end{equation}
  Therefore, the norms of the numerical solutions grow monotonically and without
  bounds.
\end{proof}

\begin{example}
  Besides SSPRK(10,4) of \cite{ketcheson2008highly},
  Theorem~\ref{thm:Lipschitz-growth} can be applied to the popular explicit
  strong stability preserving Runge-Kutta method SSPRK(3,3) of \cite{shu1988efficient}.
  For this scheme we have found other ODEs that appear to lead to unbounded growth
  of the energy;
  e.g. $L(t) = \sin(t) \begin{psmallmatrix} 0 & -1 \\ 1 & 0 \end{psmallmatrix}$
  and the example in Subsection~\ref{subsec:numerical-results-linear}.
\end{example}

Adapting a result of Burrage and Butcher~\cite{burrage1979stability}
slightly, energy stability for \eqref{eq:ode-linear} is equivalent to
algebraic stability for some schemes.
\begin{theorem}
\label{thm:Lipschitz-algebraic-stability}
  A non-confluent Runge-Kutta method (i.e. with distinct nodes $c_i$) that is energy stable for
  semibounded linear problems \eqref{eq:ode-linear} under a time step restriction
  depending on an upper bound of $\norm{L(t)}$ and the Lipschitz constant of
  $t \mapsto L(t)$ must be algebraically stable.
\end{theorem}
\shortlong{The proof of Theorem~\ref{thm:Lipschitz-algebraic-stability}
is more or less a repetition of a result of
Burrage and Butcher~\cite{burrage1979stability}, enhanced by noting
that the varying coefficient can be chosen such that $\norm{L(t)}$
and the Lipschitz constant of $t \mapsto L(t)$ are arbitrarily small.}{
\begin{proof}
  This proof is more or less a repetition of a result of
  Burrage and Butcher~\cite{burrage1979stability},  enhanced by noting
  that the varying coefficient can be chosen such that $\norm{L(t)}$
  and the Lipschitz constant of $t \mapsto L(t)$ are arbitrarily small.
  For completeness, the proof is given in the following.

  Choosing $L(t) = -\lambda(t) \in \R$, $\lambda(c_i \dt) = \epsilon$, and
  $\lambda(c_j \dt) = 0$ for $j \neq i$ with $\epsilon > 0$ sufficiently small,
  $L$ can be extended to a smooth mapping with arbitrarily small $\norm{L(t)}$
  and Lipschitz constant of $t \mapsto L(t)$. Hence, energy stability and
  \eqref{eq:estimate-RK} imply
  \begin{equation}
    \norm{u_+}^2 - \norm{u_0}^2
    =
    - 2 \epsilon \dt b_i \norm{u_i}^2 + \epsilon^2 \dt^2 (b_i^2 - 2 b_i a_{ii}) \norm{u_i}^2
    \leq
    0.
  \end{equation}
  For small $\epsilon \dt > 0$ and $u_0 \neq 0$, $u_i \neq 0$ and the second
  term is negligible. Hence, $b_i \geq 0$.

  Considering $L(t) = \lambda(t) \begin{psmallmatrix} 0 & -1 \\ 1 & 0 \end{psmallmatrix}$
  with $\lambda(t) \in \R$, \eqref{eq:estimate-RK} yields
  \begin{equation}
  \label{eq:estimate-RK-skew-sym}
    \norm{u_+}^2 - \norm{u_0}^2
    =
    \dt^2 \sum_{i,j=1}^s (b_i b_j - b_i a_{ij} - b_j a_{ji}) \lambda(c_i \dt) \lambda(c_j \dt) \scp{u_i}{u_j}.
  \end{equation}
  For $\lambda(c_i \dt) = \epsilon \xi_i$ with arbitrary $\xi \in \R^s$ and
  sufficiently small $\epsilon > 0$, $L$ can be extended to a smooth mapping with
  arbitrarily small $\norm{L(t)}$ and Lipschitz constant of $t \mapsto L(t)$.
  Since $u_i = u_0 + \O(\epsilon \dt)$, energy stability and \eqref{eq:estimate-RK-skew-sym}
  imply
  \begin{equation}
    \epsilon^2 \dt^2
    \sum_{i,j=1}^s (b_i b_j - b_i a_{ij} - b_j a_{ji}) \xi_i \xi_j \norm{u_0}^2
    + \O(\epsilon^3 \dt^3)
    \leq 0.
  \end{equation}
  Hence, the matrix with entries $(b_i b_j - b_i a_{ij} - b_j a_{ji})$ must be
  negative semidefinite.
\end{proof}}

Theorem~\ref{thm:Lipschitz-algebraic-stability} shows that conditional stability
for non-autonomous linear problems is equivalent to unconditional stability for
general nonlinear problems in the context of non-confluent Runge-Kutta methods
and ODEs with semibounded operators.

\shortlong{}{
\begin{remark}
  If a Runge-Kutta method satisfying the assumptions of
  Theorem~\ref{thm:Lipschitz-algebraic-stability} is irreducible in the sense of
  Dahlquist and Jeltsch, the weights $b_i$ must be strictly positive
  \cite{dahlquist1979generalized}.
\end{remark}}

\subsection{Numerical Results}
\label{subsec:numerical-results-linear}

The problems constructed in the proofs above are rather special and
perhaps not typical of applications.  The following example shows
that the (poor) behaviors suggested in the above theorems also occur
for a more natural problem.
Consider the linear advection equation
\begin{equation}
\label{eq:linear-advection}
\begin{aligned}
  \partial_t u(t,x) + \sin(t^2) \partial_x u(t,x) &= 0,
  && t \in (0,100), x \in [-1, 1],
  \\
  u(0,x) &= \sin(\pi x),
  && x \in [-1,1],
\end{aligned}
\end{equation}
with periodic boundary conditions. A finite difference
semidiscretization using the classical second order accurate central
stencil and 50 grid points results in a skew-symmetric ODE \eqref{eq:ode}.
The third order method SSPRK(3,3) of \cite{shu1988efficient} is given by the
Butcher tableau
\begin{equation}
\label{eq:SSPRK33}
\begin{aligned}
  \begin{array}{c|ccc}
  0 &  & &  \\
  1 & 1 & & \\
  \nicefrac{1}{2} & \nicefrac{1}{4} & \nicefrac{1}{4} & \\
  \hline
  & \nicefrac{1}{6} & \nicefrac{1}{6} & \nicefrac{2}{3} \\
  \end{array}
\end{aligned}
\end{equation}
For this method, the time step $\dt = 10^{-5}$
is approximately three orders of magnitude smaller than required for
energy stability of a corresponding constant coefficient problem
\cite{tadmor2002semidiscrete, ranocha2018L2stability}. However, the energy
$\norm{u(t)}^2 = \Delta x \sum_i u_i^2$ increases exponentially, as can
be seen in Figure~\ref{fig:linear_time_dep_advection}.
In contrast, the energy of the numerical approximation using the scheme
\eqref{eq:probably-stable-scheme} is decreasing, in accordance with
Proposition~\ref{pro:probbaly-stable-scheme-1}.

The methods are implemented using double precision numbers \texttt{Float64}
in the package \texttt{DifferentialEquations.jl}
\cite{rackauckas2017differentialequations} in Julia \cite{bezanson2017julia}.
The source code for these numerical experiments is available online
\cite{ranocha2019energyRepro}.

\begin{figure}
\centering
  \includegraphics[width=0.6\textwidth]{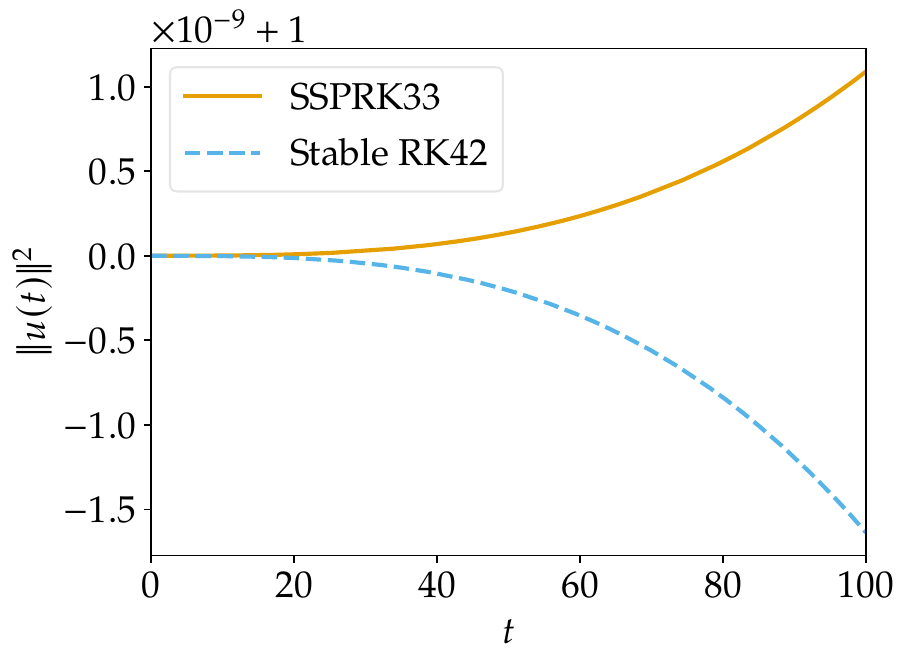}
  \caption{Discrete energies of the numerical solution of
           \eqref{eq:linear-advection} using a skew-symmetric finite
           difference operato with $\dt = 10^{-5}$.}
  \label{fig:linear_time_dep_advection}
\end{figure}

\section{Time-Independent Nonlinear Operators}
\label{sec:autonomous}

In this section, the special case of time-independent but possibly nonlinear
operators is studied. Hence,
\begin{equation}
\label{eq:ode-autonomous}
\begin{aligned}
  \od{}{t} u(t) &= f(u(t)),
  && t \in [0,T],
  \\
  u(0) &= u_0,
\end{aligned}
\end{equation}
is considered as special case of \eqref{eq:ode}.
The following result was obtained in \cite[Theorem~7.1]{ranocha2020strong}.
\begin{theorem}
\label{thm:first-order}
  There exist first order accurate explicit Runge-Kutta methods with two stages
  that are strong stability preserving and energy stable for the
  ODE \eqref{eq:ode-autonomous} with semibounded and Lipschitz continuous $f$
  with $\normLip{f} \leq L$ under a time step constraint
  $0 < \dt \le \dt_\mathrm{max} \propto L^{-1}$.
\end{theorem}
The technique used in the stability proof of Theorem~\ref{thm:first-order}
can be described as follows.
\begin{enumerate}
  \item
  Expand $\norm{u_+}^2 - \norm{u_0}^2$ as in \eqref{eq:estimate-RK}.

  \item
  Use the Lipschitz continuity of the right-hand side of \eqref{eq:ode-autonomous} in order
  to expand the $\dt^2$ term in a power series in $\dt$.

  \item
  Use the coefficients of the scheme and signs of the dominant terms in the power
  series to determine conditions for energy stability.
\end{enumerate}
The second step can be generalised to higher order schemes by considering analytic
right-hand sides $f$ and the expansion \cite[equation (313b)]{butcher2016numerical}
\begin{equation}
\label{eq:expansion-g(u_i)}
  f(u_i)
  =
  \sum_{\abs{t} \leq n} \dt^{\abs{t}-1} \frac{1}{\sigma(t)} (\Phi_i D)(t) \, F(t)(u_0)
  + \O(\dt^{n+1}).
\end{equation}
This is a sum over all trees $t$ of order $\abs{t} \leq n$, $\sigma(t)$ is the
symmetry of $t$, $(\Phi_i D)(t)$ is the $i$-th derivative weight of $t$, and
$F(t)(u_0)$ the elementary differential associated with the tree $t$ and the
right hand side $f$, evaluated at $u_0$.

Inserting \eqref{eq:expansion-g(u_i)} into \eqref{eq:estimate-RK} yields
\begin{multline}
\label{eq:estimate-RK-trees}
  \norm{u_+}^2 - \norm{u_0}^2
  =
  2 \dt \sum_{i=1}^{s} b_{i} \scp{ u_i }{ f(u_i) }
  \\
  + \sum_{\mathclap{\abs{t_1},\, \abs{t_2} \leq n}} \dt^{\abs{t_1} + \abs{t_2}}
    \sum_{i,j=1}^s (b_i b_j - b_{i} \, a_{ij} - b_j a_{ji})
                  \frac{ (\Phi_i D)(t_1) \, (\Phi_j D)(t_2)}{ \sigma(t_1) \sigma(t_2) }
    \scp{ F(t_1)(u_0) }{ F(t_2)(u_0) }
  \\
  + \O(\dt^{n+3}).
\end{multline}

\begin{example}
\label{ex:SSPRK33-expansion}
  For the three stage, third order method SSPRK(3,3) of \cite{shu1988efficient}
  given by the Butcher coefficients \eqref{eq:SSPRK33},
  the first terms of the sum over trees in \eqref{eq:estimate-RK-trees} are
  \begin{multline}
  \label{eq:SSPRK33-expansion}
    \dt^4 \left[
      \frac{1}{6} \scp{ F(\rootedtree[])(u_0) }{ F(\rootedtree[[*]])(u_0) }
      - \frac{1}{12} \scp{ F(\rootedtree[])(u_0) }{ F(\rootedtree[*,*])(u_0) }
      + \frac{1}{12} \norm{ F(\rootedtree[*])(u_0) }^2
    \right]
    + \O(\dt^5)
    \\
    =
    \dt^4 \left[
      \frac{1}{6} \scp{ f }{ f' f' f }
      - \frac{1}{12} \scp{ f }{ f''(f,f) }
      + \frac{1}{12} \norm{ f' f }^2
    \right]
    + \O(\dt^5).
  \end{multline}
  The calculations leading to this expression are available in a Jupyter
  notebook online \cite{ranocha2019energyRepro}. As an example, the coefficient
  $\nicefrac{1}{6}$ of the first term multiplied by $\dt^4$ can be obtained
  inserting the definition of the derivative weights
  \cite[Definition~312A]{butcher2016numerical} and the symmetry
  \cite[Table~301(I)]{butcher2016numerical} as
  \begin{equation}
  \begin{aligned}
    &
    2 \sum_{i,j=1}^s (b_i b_j - b_{i} \, a_{ij} - b_j a_{ji})
                  \frac{ (\Phi_i D)(\rootedtree[]) \, (\Phi_j D)(\rootedtree[[*]])}{ \sigma(\rootedtree[]) \sigma(\rootedtree[[*]]) }
    \\
    =&
    2 \sum_{i,j=1}^s (b_i b_j - b_{i} \, a_{ij} - b_j a_{ji})
                  \frac{ 1 \cdot \sum_{k,l=1}^s a_{jk} a_{kl} }{ 1 \cdot 1 }
  \end{aligned}
  \end{equation}
  Here, the factor $2$ arises because of the symmetry with respect to the
  order of the trees $\rootedtree[]$ and $\rootedtree[[*]]$.
  Inserting $\sum_{k,l=1}^s a_{jk} a_{kl} = \nicefrac{\delta_{j3}}{4}$,
  where $\delta_{j3}$ is the Kronecker delta, and evaluating the remaining
  sum yields
  \begin{equation}
    \frac{1}{2} \sum_{i}^s (b_i b_3 - b_{i} \, a_{i3} - b_3 a_{3i})
    =
    \frac{1}{2} \left( \frac{2}{3} - 0 - \frac{2}{3} \cdot \frac{1}{2} \right)
    =
    \frac{1}{6}.
  \end{equation}
  By a suitable choice of a non-dissipative right hand side $f$ such that
  $\forall u\colon \scp{u}{f(u)} = 0$, the leading order term in
  \eqref{eq:SSPRK33-expansion} can be made positive. Hence, the energy
  increases for every sufficiently small time step $\dt$.

  Of course, there are also problems for which SSPRK(3,3) is conditionally
  energy stable. For example, the right hand side $f\colon \R^2 \to \R^2$,
  $f(u) = \norm{u}^2 (-u_2, u_1)^T$, yields
  $- \frac{7}{12} \norm{u}^{10} + \O(\dt^5)$
  as first terms of the expansion \eqref{eq:estimate-RK-trees}.
\end{example}

\begin{example}
\label{ex:test-method-expansion}
  It is appealing to choose a Runge-Kutta method such that the coefficients in
  front of scalar products of elementary differentials whose sign cannot be
  controlled vanish and such that the coefficients multiplying non-negative
  terms such as $\norm{ f' f }^2$ are negative. However, it seems difficult
  to do this in a way that guarantees energy stability for all semibounded $f$.
  For example,
  \begin{equation}
  \begin{aligned}
    \begin{array}{c|ccc}
    0 & & &  \\
    \nicefrac{3}{8} & \nicefrac{3}{8} & & \\
    1 & -1 & 2 & \\
    \hline
    & \nicefrac{1}{22} & \nicefrac{8}{11} & \nicefrac{5}{22} \\
    \end{array}
  \end{aligned}
  \end{equation}
  represents a three stage, second order scheme. The corresponding first terms
  of the sum over trees in \eqref{eq:estimate-RK-trees} are
  \begin{align*}
  \stepcounter{equation}\tag{\theequation}
    &
      \dt^4 \biggl[
      - \frac{1}{11} \norm{ f' f }^2
    \biggr]
    + \dt^5 \biggl[
      - \frac{15}{176} \scp{ f' f }{ f' f' f }
      - \frac{49}{704} \scp{ f' f }{ f''(f, f) }
    \biggr]
    \\ &
    + \dt^6 \biggl[
      - \frac{45}{2816} \scp{ f' f }{ f' f''(f, f) }
      - \frac{15}{176} \scp{ f' f }{ f''(f, f' f) }
      - \frac{347}{16896} \scp{ f' f }{ f'''(f, f, f) }
      \\ &\quad\quad\quad
      + \frac{225}{7744} \norm{ f' f' f }^2
      + \frac{255}{30976} \scp{ f' f' f }{ f''(f, f) }
      - \frac{149}{30976} \norm{ f''(f, f) }^2
    \biggr]
    \\ &
    + \dt^7 \biggl[
      - \frac{45}{22528} \scp{ f' f }{ f''(f' f' f, f) }
      - \frac{45}{2816} \scp{ f' f }{ f''(f''(f, f), f) }
      \\ &\quad\quad\quad
      - \frac{45}{1408} \scp{ f' f }{ f''(f' f, f' f) }
      - \frac{15}{352} \scp{ f' f }{ f'''(f' f, f, f) }
      \\ &\quad\quad\quad
      - \frac{2641}{540672} \scp{ f' f }{ f''''(f, f, f, f) }
      + \frac{675}{61952} \scp{ f' f' f }{ f' f''(f, f) }
      \\ &\quad\quad\quad
      + \frac{225}{3872} \scp{ f' f ' f }{ f''(f' f, f) }
      + \frac{205}{22528} \scp{ f' f' f }{ f'''(f, f, f) }
      \\ &\quad\quad\quad
      + \frac{765}{495616} \scp{ f''(f, f) }{ f' f''(f, f) }
      + \frac{255}{30976} \scp{ f''(f, f) }{ f''(f' f, f) }
      \\ &\quad\quad\quad
      - \frac{1}{16896} \scp{ f''(f, f) }{ f'''(f, f, f) }
    \biggr]
    \\ &
    + \dt^8 \biggl[
      - \frac{135}{720896} \scp{ f' f }{ f' f''''(f, f, f, f) }
      - \frac{45}{22528} \scp{ f' f }{ f''(f'''(f, f, f), f) }
      \\ &\quad\quad\quad
      - \frac{135}{11264} \scp{ f' f }{ f''(f''(f, f), f' f) }
      - \frac{45}{5632} \scp{ f' f }{ f'''(f''(f, f), f, f) }
      \\ &\quad\quad\quad
      - \frac{45}{1408} \scp{ f' f }{ f'''(f' f, f' f, f) }
      - \frac{5}{352} \scp{ f' f }{ f''''(f' f, f, f, f) }
      \\ &\quad\quad\quad
      - \frac{-20723}{21626880} \scp{ f' f }{ f'''''(f, f, f, f, f) }
      + \frac{675}{495616} \scp{ f' f' f }{ f' f'''(f, f, f) }
      \\ &\quad\quad\quad
      + \frac{675}{61952} \scp{ f' f' f }{ f''(f''(f, f), f) }
      + \frac{675}{30976} \scp{ f' f' f }{ f''(f' f, f' f) }
      \\ &\quad\quad\quad
      + \frac{255}{7744} \scp{ f' f' f }{ f'''(f' f, f, f) }
      + \frac{22765}{7929856} \scp{ f' f' f }{f''''(f, f, f, f) }
      \\ &\quad\quad\quad
      + \frac{765}{3964928} \scp{ f''(f, f) }{ f' f'''(f, f, f) }
      + \frac{765}{495616} \scp{ f''(f, f) }{ f''(f''(f, f), f) }
      \\ &\quad\quad\quad
      + \frac{765}{247808} \scp{ f''(f, f) }{ f''(f' f, f' f) }
      + \frac{265}{61952} \scp{ f''(f, f) }{ f'''(f' f, f, f) }
      \\ &\quad\quad\quad
      + \frac{1667}{5947392} \scp{ f''(f, f) }{ f''''(f, f, f, f) }
      + \frac{2025}{1982464} \norm{ f' f''(f, f) }^2
      \\ &\quad\quad\quad
      + \frac{675}{61952} \scp{ f' f''(f, f) }{ f''(f' f, f) }
      + \frac{615}{360448} \scp{f' f''(f, f) }{ f'''(f, f, f) }
      \\ &\quad\quad\quad
      + \frac{225}{7744} \norm{ f''(f' f, f) }^2
      + \frac{205}{22528} \scp{ f''(f' f, f) }{ f'''(f, f, f) }
      \\ &\quad\quad\quad
      + \frac{1019}{1622016} \norm{ f'''(f, f, f) }^2
    \biggr]
    + \O(\dt^9).
  \end{align*}
  Again, the calculations leading to this expression are available in a Jupyter
  notebook online \cite{ranocha2019energyRepro}.
  If $\norm{ f' f } \neq 0$, the scheme is dissipative for sufficiently small
  time step $\dt > 0$.
  Similarly, if $\norm{ f' f } = 0$ and $\norm{ f''(f, f) } \neq 0$, the norm
  of the numerical solutions cannot increase for sufficiently small $\dt > 0$.
  If $f' f = 0$ and $f''(f, f) = 0$, the $\dt^7$ terms vanish additionally.
  However, the energy of the numerical solutions increases for arbitrarily
  small $\dt > 0$ if $f' f = 0$, $f''(f, f) = 0$ and $\norm{ f'''(f, f, f) } \neq 0$,
  since there is a positive term
  $\dt^8 \frac{1019}{1622016} \norm{ f'''(f, f, f) }^2 + \O(\dt^9)$
  and all other terms vanish.
\end{example}

\subsection{Necessary Conditions for Energy Stability}

The Examples~\ref{ex:SSPRK33-expansion} and \ref{ex:test-method-expansion}
demonstrate the importance of the bushy trees $\rootedtree[*]$,
$\rootedtree[*,*]$, $\rootedtree[*,*,*]$, $\dots$ with corresponding derivative
weights $c_i$, $c_i^2$, $c_i^3$, $\dots$ and elementary differentials $f' f$, $f''(f, f)$,
$f'''(f, f, f)$ etc. While it is known that the elementary differentials are linearly
independent for general right hand sides $f$ \cite[Section~314]{butcher2016numerical},
it is of interest to study the independence of these terms for semibounded $f$
and in particular for non-dissipative $f$.

In order to do that, the setting will be changed. Instead of considering
a Hilbert space, $\H$ is a real semi inner product space. The semi inner
product is still written as $\scp{\cdot}{\cdot}$ and $\norm{\cdot}$ denotes
the induced seminorm (i.e. nearly a norm but not necessarily definite). The
same definitions of energy stability etc. as for inner product spaces are
used. Semi inner products will be considered only in this subsection.

\begin{theorem}
\label{thm:elementary-differentials}
  For each $k \in \N$, there is an autonomous ODE with semibounded right hand side
  $f$ in a semi inner product space such that the elementary differentials evaluated
  at $u_0$ satisfy $f^{(k)}(f, \dots, f) \neq 0$ and $f^{(l)}(f, \dots, f) = 0$
  for all $l \neq k$.
  Additionally, $\scp{f}{f^{(k)}} = 0$ and $\forall u\colon \scp{u}{f(u)} = 0$.
\end{theorem}
\begin{proof}
  Consider the space $\R^3$ equipped with the semi inner product induced by the
  matrix $P = \diag(0, 1, 1)$, i.e. $\scp{u}{v}_P = u^T P v$. Consider the ODE
  \begin{equation}
    u'(t) = f(u(t)),
    \quad
    u(0) = \vect{0 \\ 1 \\ 0},
    \quad
    f(u)
    :=
    \vect{1 \\ 0 \\ 0}
    + u_1^k \vect{0 \\ -u_3 \\ u_2}.
  \end{equation}
  The right hand side $f$ satisfies $\forall u \in \R^3\colon \scp{u}{f(u)}_P = 0$
  and $f(u_0) = (1, 0, 0)^T \neq 0$.
  Hence, for $l \in \N$, the $i$-th component of the elementary differential
  $f^{(l)}(f, \dots, f)$ evaluated at $u_0$ is \cite[Definition~310A]{butcher2016numerical}
  \begin{equation}
    \left[ f^{(l)}(f, \dots, f) \right]^i
    =
    f^{i}_{j_1, \dots, j_l} f^{j_1} \dots f^{j_l}
    =
    f^{i}_{1, \dots, 1}
    \implies
    f^{(l)}(f, \dots, f)
    =
    \delta_{k,l} k! \vect{0 \\ 0 \\ 1}.
  \end{equation}
  Here, summation over repeated indices is implied, upper indices denote components,
  lower indices denote derivatives, and $\delta_{k,l}$ is the Kronecker delta.
\end{proof}

\begin{remark}
\label{rem:semi-inner-products}
  If semi inner products are allowed, problems depending explicitly on time
  have to be considered again. Indeed, the test problem constructed in the
  proof of Theorem~\ref{thm:elementary-differentials} is of this form since
  $u_1$ can be interpreted as time $t$ and the numerical approximation of
  $u_1$ is equal to $t$ if the classical condition $c_i = \sum_{j=1}^s a_{ij}$
  is satisfied.

  Hence, the results of section~\ref{sec:linear} can be applied, showing
  that there is nothing to gain if there is a $c_i$ distinct from the
  others with $b_i \neq 0$.
  Thus, it is interesting whether the independence of the elementary
  differentials for energy conservative right hand sides holds also
  in inner product spaces. Since this problem seems to be intractable with
  the current methods, it is left for future investigations.
\end{remark}

Theorem~\ref{thm:elementary-differentials} shows that the choice of elementary
differentials made in Example~\ref{ex:test-method-expansion} is possible. Hence, the method
mentioned there is not energy stable for general autonomous and semibounded
problems. The basic argument used there can be formulated as follows.
\begin{theorem}
\label{thm:strong-stability-impossible-basic}
  Consider a Runge-Kutta method with order of accuracy at least two.
  If there is a $k \in \N$ such that the Butcher coefficients satisfy
  \begin{equation}
  \label{eq:sign-condition}
    \sum_{i,j=1}^s (b_i b_j - b_{i} \, a_{ij} - b_j \, a_{ji}) c_i^k c_j^k > 0,
  \end{equation}
  then the method is not energy stable for general autonomous and semibounded
  ODEs in semi inner product spaces.
\end{theorem}
\begin{proof}
  Since the method is at least second order accurate, the lowest order term
  in the sum involving trees in \eqref{eq:estimate-RK-trees} vanishes since
  \begin{equation}
  \begin{aligned}
    \sum_{i,j=1}^s (b_i b_j - b_{i} \, a_{ij} - b_j a_{ji})
    (\Phi_i D)(\rootedtree[]) \, (\Phi_j D)(\rootedtree[])
    &=
    \sum_{i,j=1}^s (b_i b_j - b_{i} \, a_{ij} - b_j a_{ji})
    \\
    &=
    \left( \sum_{i=1}^s b_i \right)^2 - 2 \sum_{i=1}^s b_i c_i
    =
    1 - 2 \cdot \frac{1}{2}
    =
    0.
  \end{aligned}
  \end{equation}
  Because of Theorem~\ref{thm:elementary-differentials}, it is possible to
  choose $f$ such that the remaining terms all vanish except the one corresponding
  to the bushy tree with $k$ leaves.
  While this is not formulated directly there, a close inspection of the proof
  reveals that this is indeed true. For example, one can choose $f$ so that $f' = 0$
  and $\scp{ f }{ f' f''(f, f) } = 0$ while $f \neq 0$ and $f''(f, f) \neq 0$.

  By choosing such an $f$ all terms of order up to $\dt^{2k+1}$ vanish and
  \eqref{eq:estimate-RK-trees} takes the form
  \begin{equation}
    \norm{u_+}^2 - \norm{u_0}^2
    =
    \dt^{2k+2} \underbrace{\sum_{i,j=1}^s (b_i b_j - b_{i} \, a_{ij} - b_j a_{ji}) c_i^k c_j^k
    \norm{f^{(k)}(f, \dots, f)}^2}_{> 0}
    + \O(\dt^{2k+3}).
  \end{equation}
  Hence, $\norm{u_+}^2 > \norm{u_0}^2$ for arbitrarily small $\dt > 0$.
\end{proof}

Theorem~\ref{thm:strong-stability-impossible-basic} implies in particular that
$\sum_{i,j=1}^s (b_i b_j - b_{i} \, a_{ij} - b_j a_{ji}) c_i^k c_j^k$ must be
non-positive for B-stable methods. This is related to algebraically stable
schemes, where the matrix with entries $(b_i b_j - b_{i} \, a_{ij} - b_j a_{ji})$
is negative semidefinite. It can be proved that B-stable schemes are indeed
algebraically stable if certain additional (technical) assumptions are satisfied,
e.g. if the method is non-confluent \cite[Corollary~IV.12.14]{hairer2010solving}
or irreducible \cite[Theorem~IV.12.18]{hairer2010solving}.

\begin{example}
  For the B-stable but not algebraically stable reducible schemes with
  $A = \begin{pmatrix} 1/2 & 0 \\ 0 & 1/2 \end{pmatrix}$, $b = (2, -1)^T$
  \cite[page~80]{crouzeix1979stabilite} and
  $A = \begin{pmatrix} 1/2 & 0 \\ 1/4 & 1/4 \end{pmatrix}$, $b = (1/2, 1/2)^T$
  \cite[Table~IV.12.2]{hairer2010solving},
  the terms $\sum_{i,j=1}^s (b_i b_j - b_{i} \, a_{ij} - b_j a_{ji})
  (\Phi_i D)(t_1) \, (\Phi_j D)(t_2)$ vanish, exactly as for the implicit midpoint
  method to which these schemes can be reduced.
\end{example}

While explicit Runge-Kutta methods cannot be algebraically stable,
which would imply unconditional energy stability for all semibounded problems,
it is interesting to study whether
they can be stable for these problems under a suitable time step restriction.
Theorem~\ref{thm:first-order} proves that there are indeed conditionally
energy stable schemes of first order. For higher order schemes, it remains to
check the condition given in Theorem~\ref{thm:strong-stability-impossible-basic}.
While this can be done for every scheme given explicitly, there are some results
for general classes of schemes.
\begin{theorem}
\label{thm:sign-condition}
  Consider an explicit Runge-Kutta method. Assume that there is a unique $\imax$
  such that $\abs{c_\imax} = \max_i \abs{c_i}$ and that $b_\imax \neq 0$. Then
  there is a $k \in \N$ such that \eqref{eq:sign-condition} is satisfied.
  Hence, if the method is at least second order accurate, it is not energy
  stable for general autonomous and semibounded ODEs in semi inner product
  spaces.
\end{theorem}
\begin{proof}
  The expression on the left hand side of \eqref{eq:sign-condition} can be written as
  \begin{equation}
    \bigl( b^T c^k \bigr)^2 - 2 \bigl( c^k \bigr)^T \diag(b) A c^k,
  \end{equation}
  where the exponentiation is performed componentwise. Using the given assumptions,
  \begin{equation}
    \left( \frac{c}{c_\imax} \right)^k \to e_\imax,
    \quad k \to \infty,
  \end{equation}
  where $e_\imax$ is the standard unit vector with components
  $(e_\imax)_j = \delta_{j,\imax}$. Since the Runge-Kutta scheme is explicit, $A$
  is a strictly lower triangular matrix and $e_\imax^T \diag(b) A e_\imax = 0$.
  Because of $b^T e_\imax = b_\imax \neq 0$,
  \begin{equation}
    \biggl( \frac{ b^T c^k }{ c_\imax } \biggr)^2
    - 2 \frac{ (c^k)^T }{ c_\imax^k } \diag(b) A \frac{ c^k }{ c_\imax^k }
    \to
    \bigl( b^T e_\imax \bigr)^2 - 2 e_\imax^T \diag(b) A e_\imax^k
    =
    b_\imax^2
    \neq 0,
  \end{equation}
  for $k \to \infty$. Hence, there is a $k \in \N$ such that \eqref{eq:sign-condition}
  is satisfied.
\end{proof}

\begin{remark}
  Theorem~\ref{thm:sign-condition} can also be applied to many confluent methods
  such as the ten-stage, fourth order, explicit strong stability preserving method
  SSPRK(10,4) of \cite{ketcheson2008highly}. Indeed, $\imax = 10$, $c_\imax = 1$,
  and $b_\imax = \frac{1}{10}$ in that case. That this scheme is not energy
  stable for autonomous and semibounded problems has also been proved using some
  specific counterexamples in \cite[Sections~4.3 and 6]{ranocha2020strong}.
\end{remark}

\begin{remark}
  The argument used to prove Theorem~\ref{thm:sign-condition} can also be applied
  to implicit Runge-Kutta methods with $a_{\imax,\imax} = 0$.
\end{remark}

\subsection{Sufficient Conditions for Energy Stability}

Theorem~\ref{thm:sign-condition} does not imply that all explicit
Runge-Kutta methods of order two or greater cannot be energy stable.
\begin{example}
\label{ex:sign-condition}
  The Runge-Kutta method with Butcher coefficients
  \begin{equation}
  \begin{aligned}
    \begin{array}{c|cccc}
    0 & & & &  \\
    1 & 1 & & & \\
    0 & 1 & -1 & & \\
    1 & -1 & 1 & 1 & \\
    \hline
    & \nicefrac{1}{4} & \nicefrac{1}{4} & \nicefrac{1}{4} & \nicefrac{1}{4} \\
    \end{array}
  \end{aligned}
  \end{equation}
  has four stages and is second order accurate. Since $c_2 = 1 = c_4$, it does
  not satisfy the assumptions of Theorem~\ref{thm:sign-condition}. Because
  $c_i \in \set{0, 1}$, we have $c^k = c$ for $k \in \N$. Hence, it suffices to consider
  $k = 1$ in \eqref{eq:sign-condition}, i.e.
  \begin{equation}
    \sum_{i,j=1}^s (b_i b_j - b_{i} \, a_{ij} - b_j \, a_{ji}) c_i c_j
    =
    - \frac{1}{4}
    <
    0.
  \end{equation}
  Thus, the sum in \eqref{eq:sign-condition} is negative for all $k \in \N$.
\end{example}

\begin{theorem}
\label{thm:probably-stable-scheme}
  The Runge-Kutta method given in Example~\ref{ex:sign-condition} is conditionally energy
  stable with respect to autonomous ODEs \eqref{eq:ode-autonomous} with analytical and
  semibounded right hand side.
\end{theorem}
\begin{proof}
  Since the right hand side $f$ is analytical, the energy difference after one
  time step can be expanded as in \eqref{eq:estimate-RK-trees}. Because of the
  semiboundedness of $f$, the term proportional to $\dt$ is non-positive.

  There are no inner products between $f = F(\rootedtree[])$ and higher
  order elementary differentials in the remaining terms because
  \begin{equation}
    \sum_{i,j=1}^s (b_i b_j - b_i a_{ij} - b_j a_{ji})
      \frac{(\Phi_i D)(\rootedtree[]) \, (\Phi_j D)(t_2)}{\sigma(\rootedtree[]) \, \sigma(t_2)}
    =
    \sum_{j=1}^s
      \underbrace{\biggl( b_j - \sum_{i=1}^s b_i a_{ij} - b_j c_j \biggr)}_{= 0}
      \frac{(\Phi_j D)(t_2)}{\sigma(t_2)},
  \end{equation}
  which can be computed explicitly.
  The first remaining terms are of the form
  \begin{equation}
  \begin{aligned}
    &
    - \frac{1}{4} \dt^4 \norm{f' f}^2
    + \dt^5 \biggl[
      - \frac{1}{2} \scp{ f' f }{ f' f' f }
      - \frac{1}{4} \scp{ f' f }{ f''(f, f) }
    \biggr]
    \\&
    + \dt^6 \biggl[
        \frac{1}{4} \scp{ f' f }{ f' f' f' f }
      - \frac{1}{4} \scp{ f' f }{ f' f''(f, f) }
      - \frac{1}{4} \scp{ f' f }{ f''(f' f, f) }
      \\&\qquad\quad
      - \frac{1}{12} \scp{ f' f }{ f'''(f, f, f) }
      + \frac{1}{2} \norm{ f' f' f }^2
      - \frac{1}{4} \scp{ f' f' f }{ f''(f, f) }
      - \frac{1}{16} \norm{ f''(f, f) }^2
    \biggr].
  \end{aligned}
  \end{equation}
  If $f' f \neq 0$, the $\dt^4$ term dominates the other ones for sufficiently
  small $\dt > 0$ and the scheme is energy stable.
  If $f' f = 0$ and $f''(f, f) \neq 0$, the $\dt^4$, $\dt^5$, and most of the
  $\dt^6$ terms vanish. Only $- \frac{1}{16} \dt^6 \norm{ f''(f, f) } < 0$ remains
  and dominates higher order terms, resulting in a stable scheme for small
  $\dt > 0$.

  This argument can be applied similarly to all other terms. Suppose that
  $f' f = \dots = f^{(k-1)}(f, \dots, f) = 0$ and $f^{(k)}(f, \dots, f) \neq 0$.
  The terms up to (and including) $\O(\dt^{2k+1})$ vanish,
  since (as described at the beginning of the proof) for this method,
  the series \eqref{eq:estimate-RK-trees} does not include any terms
  involving an inner product of $f$ and other elementary differentials.
  Most of the $\dt^{2k+2}$ terms
  vanish too, except the one proportional to $\norm{ f^{(k)}(f, \dots, f) }$.
  Because $\sum_{i,j=1}^s (b_i b_j - b_i a_{ij} - b_j a_{ji}) c_i^k c_j^k < 0$,
  this term is negative and dominates higher order terms. Hence, the scheme is
  energy stable for sufficiently small $\dt > 0$.
\end{proof}

The key ingredients of the proof of Theorem~\ref{thm:probably-stable-scheme}
are distilled in
\begin{proposition}
\label{prop:rk-2-3}
  Consider a second or third order accurate explicit Runge-Kutta method
  satisfying
  \begin{equation*}
  \begin{aligned}
     &\bullet \forall i \in \set{1, \dots, s}\colon b_i \geq 0,
    &&\bullet \forall j \in \set{1, \dots, s}\colon b_j - \sum_{i=1}^s b_i a_{ij} - b_j c_j = 0,
    \\
     &\bullet \forall i \in \set{1, \dots, s}\colon c_i = \sum_{j=1}^s a_{ij},
    &&\bullet\forall k \in \N\colon \text{The sum in \eqref{eq:sign-condition} is negative}.
  \end{aligned}
  \end{equation*}
  Such a scheme is conditionally energy stable for any autonomous
  ODE \eqref{eq:ode-autonomous} with right hand side that is
  analytical and semibounded with respect to an inner product.
\end{proposition}
The proof is basically the same as that of
Theorem~\ref{thm:probably-stable-scheme}. The restriction to second and third
order accurate schemes is explained in Section~\ref{sec:limitations-HO-schemes}
below.

\begin{remark}
\label{rem:rk-2-3}
  Since infinitely many constraints have to be satisfied to apply
  Proposition~\ref{prop:rk-2-3}, it is useful to consider additional
  simplifying assumptions/constraints in order to find feasible solutions.
  The Runge-Kutta method given in Example~\ref{ex:sign-condition} and other
  schemes have been constructed using the following additional steps.
  \begin{itemize}
    \item
    Because of Theorem~\ref{thm:sign-condition}, the node with biggest
    absolute value should appear at least twice.
    In order to facilitate the search for a solution, it has been useful
    to choose the nodes $c_i$ manually. Here, the biggest node is
    chosen as $1$ (twice).
    In numerical experiments, it seemed to be useful/necessary to specify
    also the node $0$ twice.

    \item
    The sum in \eqref{eq:sign-condition} with $k = 1$ has to be negative.
    Additionally, the same sum involving only the nodes $0$ and $1$ should
    be negative. In numerical experiments, this additional condition has often
    be sufficient to guarantee that the sum in \eqref{eq:sign-condition} is
    negative for general $k \in \N$.
  \end{itemize}
\end{remark}

\begin{example}
\label{ex:stable-rk_3_5-Mathematica}
  The third order accurate Runge-Kutta method with Butcher coefficients
  \begin{equation}
  \begin{aligned}
    \begin{array}{c|ccccc}
    0 & & & & & \\
    \nicefrac{1}{2} & \nicefrac{1}{2} & & & \\
    1 & 1 & & & \\
    0 & 1 & 0 & -1 & \\
    1 & -3 & 2 & 1 & 1 \\
    \hline
    & 0 & \nicefrac{2}{3} & 0 & \nicefrac{1}{6} & \nicefrac{1}{6} \\
    \end{array}
  \end{aligned}
  \end{equation}
  has been constructed using the approach just described.
  It is energy stable for autonomous ODEs \eqref{eq:ode-autonomous}
  with analytical and semibounded right hand side in inner product spaces
  if the time step $\dt > 0$ is sufficiently small.
  Indeed, since
  \begin{equation}
    \sum_{i,j=1}^s (b_i b_j - b_{i} \, a_{ij} - b_j \, a_{ji}) c_i^k c_j^k
    =
    - \frac{11}{36} - \frac{4}{9} 2^{-k} \bigl( 1 - 2^{-k} \bigr)
    <
    0,
  \end{equation}
  the sum in \eqref{eq:sign-condition} is negative for general $k \in \N$
  and Proposition~\ref{prop:rk-2-3} can be applied.
\end{example}

\subsection{Limitations for Higher Order Schemes}
\label{sec:limitations-HO-schemes}

The conditions listed in Proposition~\ref{prop:rk-2-3} are not sufficient
to create energy stable fourth order methods, since the coefficient of
$\norm{f' f}^2$ in the expansion \eqref{eq:estimate-RK-trees} vanishes
because of accuracy constraints. However, terms like $\scp{ f' f }{ f' f' f}$ etc. appear later, which cannot be controlled in general. Hence, one has
to impose additionally that all scalar products of $f' f$ with higher order
differentials vanish. Because of
\begin{multline}
  \sum_{i,j=1}^s (b_i b_j - b_i a_{ij} - b_j a_{ji})
    \frac{(\Phi_i D)(\rootedtree[*]) \, (\Phi_j D)(t_2)}{\sigma(\rootedtree[*]) \, \sigma(t_2)}
  \\
  =
  \sum_{j=1}^s
    \biggl( b_j \sum_{i=1}^s c_i - \sum_{i=1}^s b_i c_i a_{ij} - b_j \sum_{i=1}^s a_{ji} c_i \biggr)
    \frac{(\Phi_j D)(t_2)}{\sigma(t_2)},
\end{multline}
this additional constraint is
\begin{equation}
  b_j \sum_{i=1}^s c_i - \sum_{i=1}^s b_i c_i a_{ij} - b_j \sum_{i=1}^s a_{ji} c_i = 0.
\end{equation}
By summing over $j$ and using the order conditions for order 3, it becomes
clear that some nodes $c_i$ must be negative if this constraint should be
satisfied.

Nevertheless, even this additional constraint does not suffice to guarantee
energy stability. Indeed, terms involving higher order differentials of the
form $\norm{f' f' f}^2$ appear and cannot be controlled by the previous terms
with negative coefficients.

\section{Unconditional Stability of Explicit Runge-Kutta Discretizations}
\label{sec:special-example}

In this section we investigate the possibility of obtaining unconditional
stability with explicit Runge-Kutta methods.  It is usually true in numerical
analysis that explicit methods can be only conditionally stable.  The following
(unsurprising) theorem confirms this view, in the context of the linear
autonomous initial-value problem:
\begin{align} \label{linear-ODE}
\begin{aligned}
  \od{}{t} u(t) &= L u(t),
  && t \in [0,T],
  \\
  u(0) &= u_0,
\end{aligned}
\end{align}
\begin{theorem}
\label{thm:linear-blowup}
  Let an at least first order accurate explicit Runge-Kutta method and
  constant (possibly complex) matrix $L \ne 0$ be given.
  Then, there exist initial values $u_0$ and a step size $\dt_*(A,b,L)$
  such that the numerical solution $u^n$ of \eqref{linear-ODE}
  blows up as $n \to \infty$ for any $\dt > \dt_*$.
\end{theorem}
\begin{proof}
  An $s$-stage explicit Runge-Kutta method applied to \eqref{linear-ODE}
  gives the solution $u^{n+1} = R(\dt L) u^n$, where
  $R(z) = \sum_{j=0}^d \alpha_j z^j$ is a polynomial with degree $d \le s$
  and $\alpha_0 = \alpha_1 = 1$.
  Two cases can occur.

  1. $L$ has an eigenvalue $\lambda \neq 0$ with eigenvector $u_0$:
  Then, $R(\dt L) u_0 = (I + \alpha_1 \dt L + \dots + \alpha_d \dt^d L^d) u_0
  = (I + \alpha_1 \dt \lambda + ... + \alpha_d \dt^d \lambda^d) u_0$.
  Thus, $R(\dt L)^n u_0 \to \infty$ for $n \to \infty$ if $\dt$ is big
  enough.

  2. Zero is the only eigenvalue of $L$:
  Then, there is a vector $u_0$ such that $L v \neq 0$ but $L^2 v = 0$
  (consider the Jordan canonical form). Thus,
  $R(\dt L)^n u_0 = (1 + n \alpha_1 \dt L) u_0 \to \infty$ for $n \to \infty$
  if $\dt$ is big enough.
\end{proof}
In practice, due to rounding errors, it is reasonable to expect a blowup
for almost all initial data.

\subsection{Nonlinear Problems}

It is natural to ask if a result like Theorem \ref{thm:linear-blowup} holds when $L$
is allowed to be nonlinear.  It seems quite natural to expect that the answer
is yes.  However, we have the following result:
\begin{theorem}
\label{thm:midpoint-stability}
    There exists an explicit Runge-Kutta method and non-trivial
    function $f(u)$ such that the numerical solution of $u'(t)=f(u)$
    remains bounded as $n\to\infty$ for every step size $\dt$ and
    for every initial value.
\end{theorem}
\begin{proof}
We prove this result by constructing an example -- the only example of which we
are presently aware.  We take the explicit midpoint Runge-Kutta method
\begin{equation} \label{rk2}
    y_2 = u^n + \frac{\dt}{2}f(u^n),
    \qquad
    u^{n+1} = u^n + \dt f(y_2),
\end{equation}
and the ODEs
\begin{equation} \label{ode}
  \begin{pmatrix} u_1'(t) \\ u_2'(t) \end{pmatrix}
  =
  \frac{1}{u_1^2 + u_2^2} \begin{pmatrix} -u_2 \\ u_1 \end{pmatrix}.
\end{equation}
Direct calculation of the change in energy over a step (using
\eqref{eq:estimate-RK}) reveals that it is constant:
\begin{align}
  \norm{ u_+ }^2 - \norm{ u_0 }^2 = \scp{f(u_0)}{f(y_2)} - \scp{f(y_2)}{f(y_2)} = 0.
\end{align}
\end{proof}
In fact, we can write explicitly the solution obtained for the example in the
proof.  For the general initial value $u_0 = r_0 (\cos(\theta_0), \sin(\theta_0))^T$,
the numerical solution of \eqref{ode} obtained with the explicit midpoint RK method is
\begin{align}
  u^n & = r_0 \begin{pmatrix} \cos(\theta_0+n\theta_h) \\ \sin(\theta_0+n\theta_h) \end{pmatrix},
  \quad \text{where } \quad
  \theta_h = \arccos\frac{r^2-\frac{\dt^2}{4r^2}}{r^2+\frac{\dt^2}{4r^2}}.
\end{align}

This example is quite remarkable, and naturally leads one to wonder if
others like it exist. If we assume the numerical energy is constant, then the
problem must also be energy conservative
since the Runge-Kutta scheme converges to the analytical solution (if the
right hand side is locally Lipschitz continuous). In $\R^2$, every energy
conservative problem is of the form
\begin{align} \label{fdef}
  \begin{pmatrix} u_1'(t) \\ u_2'(t) \end{pmatrix}
  =
  g(u) \begin{pmatrix} -u_2 \\ u_1 \end{pmatrix},
\end{align}
where $g$ is a scalar valued function. We have the following uniqueness result.
\begin{theorem}
\label{thm:rk2-and-problem-unique}
  Let a consistent two-stage explicit Runge-Kutta method and a
  function $g\colon \R^2 \to \R$ be given.
  Consider the ODE \eqref{fdef} and suppose that the numerical solution
  satisfies $\norm{u^n} = \norm{u_0}$ for all step sizes $\dt$ and all $n$.
  \begin{enumerate}[label=\alph*)]
    \item \label{itm:rk2-unique}
    If $g$ is not identically zero, then the Runge-Kutta method must be
    the explicit midpoint method \eqref{rk2}.

    \item \label{itm:problem-unique}
    If $g\colon \R^2 \setminus \set{0} \to \R$ is analytic,
    then $g$ must be a scalar multiple of $u \mapsto \norm{u}^{-2}$.
  \end{enumerate}
\end{theorem}
\begin{proof}
  \ref{itm:rk2-unique}
  For brevity, let $r = g(u_0)$ and $s=g(y_2) = g(u_0 + a_{21} \dt f(u_0))$.
  \begin{equation}
  \begin{aligned}
      \scp{f_1}{f_1} & = \norm{u_0}^2 r^2, \\
      \scp{f_2}{f_2} & = \norm{u_0}^2 s^2 (1 + a_{21}^2 \dt^2 r^2), \\
      \scp{f_1}{f_2} & = \norm{u_0}^2 r s.
  \end{aligned}
  \end{equation}
  Thus energy is conserved if and only if (after dividing through by
  $\norm{u_0}^2$)
  \begin{equation} \label{s-quadratic}
    2b_2 a_{21} r s
    =
    b_1^2 r^2 + b_2^2 s^2 (1 + a_{21}^2 \dt^2 r^2) + 2 b_1 b_2 r s
    =
    0.
  \end{equation}
  This is a quadratic equation in $s$, which has real roots only if
  \begin{equation}
    b_2^2 r^2 (b_1-a_{21})^2 \geq b_1^2 b_2^2 r^2 (1+a_{21}^2 \dt^2 r^2).
  \end{equation}
  This cannot hold for all $\dt$ unless the term involving $\dt$ vanishes.
  So we must have
  $b_1 b_2 a_{21} r = 0$.  The case $r=0$ is ruled out by assumption, while $b_2=0$
  implies $r=0$ by \eqref{s-quadratic}.  Taking $a_{21}=0$ implies (by \eqref{s-quadratic})
  that the ratio $r/s$ is equal to a constant independent of $\dt$ or $u$, which is not
  possible.  Thus we must have $b_1=0$; consistency then requires $b_2=1$.
  This implies that $s=0$ (trivial) or that
  \begin{align} \label{g-condition}
      g(u+a_{21} \dt f(u)) (1+a_{21}^2 \dt^2 (g(u))^2) = 2 a_{21} g(u).
  \end{align}
  Considering $\dt \to 0$, we find that necessarily $a_{21}=1/2$;
  the resulting method is the midpoint Runge-Kutta method \eqref{rk2}.

  \ref{itm:problem-unique}
  It suffices to consider $g \neq 0$. Expanding \eqref{g-condition} with
  $a_{21} = \nicefrac{1}{2}$ as required by \ref{itm:rk2-unique} yields
  \begin{equation}
    \sum_{k \geq 0} \frac{1}{k!} 2^{-k} \dt^k g^{(k)}(f, \dots, f)
    \bigl( 1 + 2^{-2} \dt^2 g^2 \bigr)
    =
    g.
  \end{equation}
  This is equivalent to
  \begin{equation}
    \frac{1}{2} \dt g' f + \sum_{k \geq 2} 2^{-k} \dt^k \left(
      \frac{1}{k!} g^{(k)}(f, \dots, f) + \frac{1}{(k-2)!} g^2 g^{(k-2)}(f, \dots, f)
    \right) = 0.
  \end{equation}
  Since this has to hold for all $\dt$,
  \begin{equation}
  \label{eq:conditions-g}
    g'f = 0, \quad \forall k \geq 2\colon
    g^{(k)}(f, \dots, f) = - \frac{k!}{(k-2)!} g^2 g^{(k-2)}(f, \dots, f).
  \end{equation}
  This infinite set of conditions determines $g$ uniquely up to a scalar
  multiple.
  Indeed, using polar coordinates $(r,\theta)$ for $u$, the condition
  $g' f = 0$ implies that $g$ does not depend on the angle~$\theta$, i.e.
  that $g$ is radially symmetric. Hence, $g$ can be considered as a function
  depending only on the radius $r$ in \eqref{g-condition}. Expanding
  this analytic function $(0,\infty) \to \R$ in \eqref{g-condition},
  all derivatives at an arbitrary point are fixed. Hence, $g$ is determined
  up to a multiplicative factor.
\end{proof}

\begin{remark}
  Theorem~\ref{thm:rk2-and-problem-unique} holds also for $\R^3$
  instead of $\R^2$. Indeed, the action of an arbitrary skew-symmetric
  matrix is equivalent to the cross product with an associated
  vector in $\R^3$. The span of this vector is irrelevant for the
  considered problem.
\end{remark}

\subsection{Numerical Experiments}

Here, some numerical experiments using the ODE \eqref{ode} are performed.
The explicit midpoint method \eqref{rk2}, the third order strong stability
preserving method SSPRK33 of \cite{shu1988efficient}, and the energy
stable second and third order methods given in Examples~\ref{ex:sign-condition}
and \ref{ex:stable-rk_3_5-Mathematica} are applied with constant time steps
$\dt \in \set{10^{-1}, 10^{-5}}$.
The methods are implemented using quadruple precision numbers \texttt{Float128}
in the package \texttt{DifferentialEquations.jl}
\cite{rackauckas2017differentialequations} in Julia \cite{bezanson2017julia}.
The source code for these numerical experiments is available at
\cite{ranocha2019energyRepro}.

The results displayed in Figure~\ref{fig:special_problem} confirm the
analytical results: The energy grows monotonically for SSPRK33, stays constant
for the midpoint rule and decays for the energy stable methods.

\begin{figure}
\centering
  \begin{subfigure}[b]{\textwidth}
    \centering
    \includegraphics[width=\textwidth]{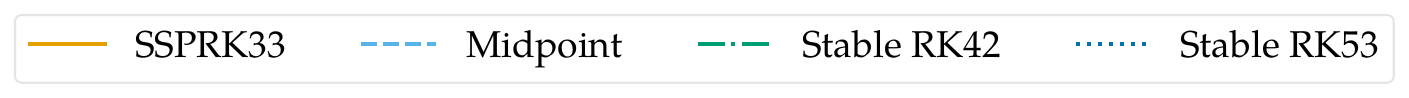}
  \end{subfigure}%
  \\
  \begin{subfigure}[b]{0.5\textwidth}
    \centering
    \includegraphics[width=\textwidth]{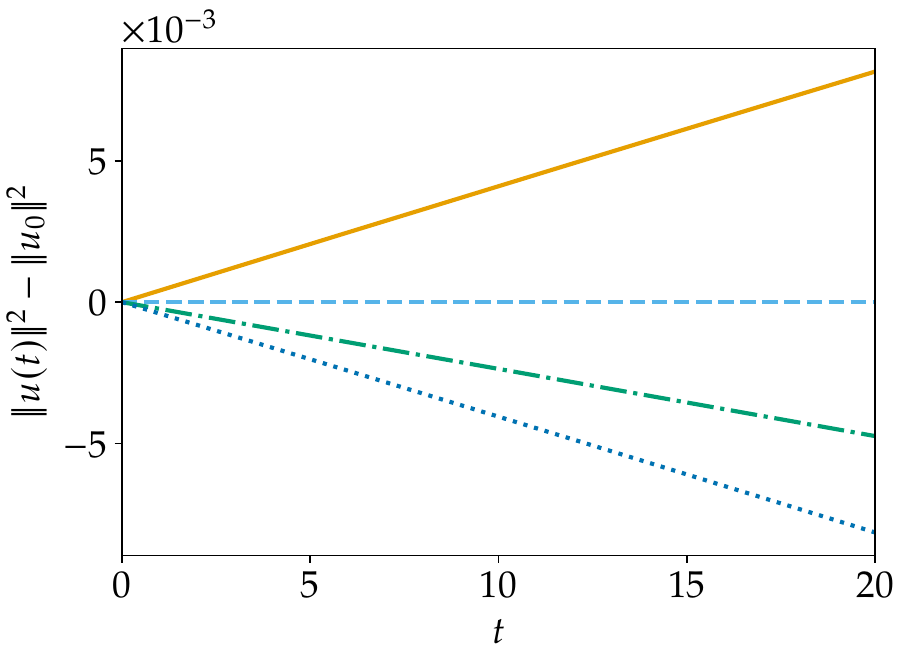}
    \caption{$\dt = 10^{-1}$.}
  \end{subfigure}%
  \begin{subfigure}[b]{0.5\textwidth}
    \centering
    \includegraphics[width=\textwidth]{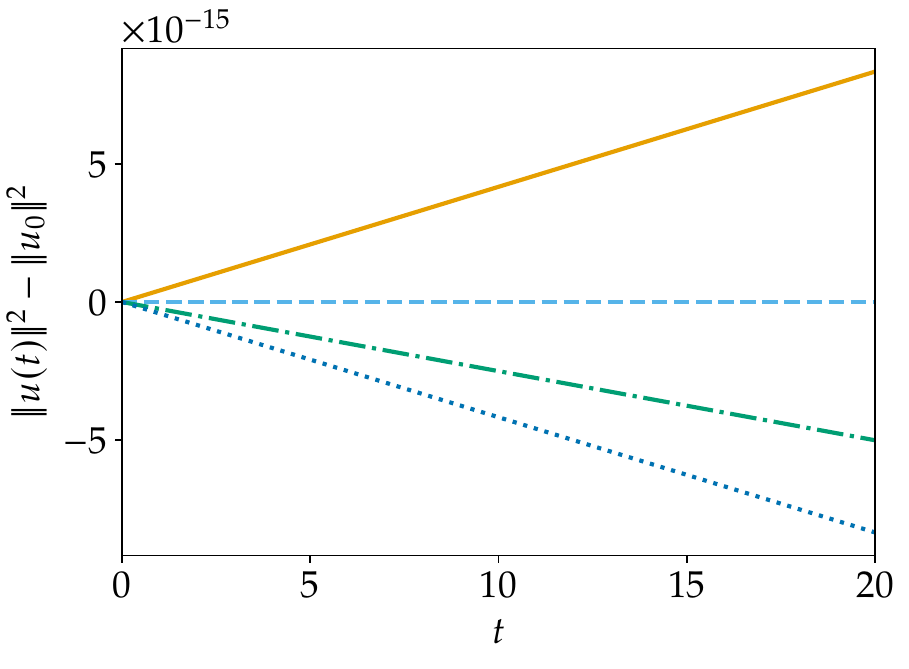}
    \caption{$\dt = 10^{-5}$.}
  \end{subfigure}%
  \caption{Evolution of the energy of numerical solutions for the
           nonlinear problem \eqref{ode}.}
  \label{fig:special_problem}
\end{figure}

\begin{remark}
  The first terms of the expansion \eqref{eq:estimate-RK-trees}
  for the test problem \eqref{ode} and the second order method of
  Example~\ref{ex:sign-condition} are
  \shortlong{$- \frac{1}{4} \dt^4 \norm{u}^{-6} + \O(\dt^5)$.}{\begin{equation}
    - \frac{1}{4} \dt^4 \norm{f' f}^2 + \O(\dt^5)
    =
    - \frac{1}{4} \dt^4 \norm{u}^{-6} + \O(\dt^5).
  \end{equation}}
  Hence, it can be verified easily that the method is conditionally
  energy stable. Similarly, the first terms for the third order
  method of Example~\ref{ex:stable-rk_3_5-Mathematica} are
  \shortlong{$- \frac{5}{12} \dt^4 \norm{u}^{-6} + \O(\dt^5)$.}{\begin{equation}
    - \frac{5}{12} \dt^4 \norm{f' f}^2 + \O(\dt^5)
    =
    - \frac{5}{12} \dt^4 \norm{u}^{-6} + \O(\dt^5).
  \end{equation}}
  Thus, this method is energy stable, too.
  Additionally, this explains why the third order method is more
  (nearly twice as) dissipative than the second order one in
  Figure~\ref{fig:special_problem}.
\end{remark}

\section{Summary and Conclusions}
\label{sec:summary}

As we have seen, explicit Runge-Kutta methods that are conditionally energy
stable for \emph{all} nonlinear autonomous semibounded problems
(with at least Lipschitz continuous right-hand sides)
are rare, but they do exist.  The existence of explicit energy stable methods
for non-autonomous problems (even in the linear setting) is still an open question.
Any explicit energy stable method for non-autonomous
problems must be confluent, since otherwise it would need to be unconditionally
stable (see Theorem~\ref{thm:Lipschitz-algebraic-stability}).
Nevertheless, there is at least a second order accurate scheme that is
energy stable for a restricted class of relevant problems
(see Proposition~\ref{pro:probbaly-stable-scheme-1}).

For nonlinear autonomous problems, our analysis is based on the series
expansion \eqref{eq:estimate-RK-trees} of the change in energy.
Besides deriving necessary conditions, Proposition~\ref{prop:rk-2-3}
and Remark~\ref{rem:rk-2-3} list sufficient conditions and approaches
that can be used to create energy stable second and third order
methods (see Theorem~\ref{thm:probably-stable-scheme} and
Example~\ref{ex:stable-rk_3_5-Mathematica}).
This approach could also be used to construct methods that are energy stable
for a particular problem, if one knows which elementary differentials of $f$
vanish.

Some of our results seem at first glance surprising or counter-intuitive.
A common intuition is that linear problems are \emph{easier} than
nonlinear problems.  But if explicit time dependence is allowed, the
linear setting becomes much more challenging: notice that the
derived necessary conditions for energy stability in this case are
more restrictive than for autonomous nonlinear problems.
In a similar vein, a \emph{particular} nonlinear ODE may be easier to deal
with than even any linear autonomous ODE, as demonstrated by
Theorem~\ref{thm:midpoint-stability} (showing unconditional stability for
a specific nonlinear ODE and explicit RK method) and Theorem~\ref{thm:linear-blowup}
(showing that no explicit RK method is unconditionally stable for any non-trivial linear
problem).

In the literature, the assumption of non-confluence is often merely technical
and not necessary.  In contrast, in our study of energy stability
for non-autonomous problems we have found that confluence is an important
property, and that certain confluent Runge-Kutta methods
(more specifically, methods with no distinct quadrature node) can have stability
properties that are impossible for non-confluent methods.

While developing many answers, this article has also revealed several open
questions and directions of further research.
First of all, an obvious question concerns the possibility of fourth or
higher order explicit Runge-Kutta that are conditionally energy stable (at least for
autonomous problems).
From a practical point of view, it would be interesting to perform a
computational optimization of energy stable Runge-Kutta methods and compare
them to state of the art schemes that are not energy stable.

More theoretically interesting questions concern the independence of
the elementary differentials for conservative problems and the existence
or uniqueness of unconditionally stable explicit Runge-Kutta methods and
associated nonlinear right hand sides.  For instance, if it were possible
to choose the values of the elementary differentials independently
while also choosing $f$ to conserve energy, then it would be possible
(in principle) to construct, for each Runge-Kutta method (including all
explicit methods), a problem for which that method is unconditionally energy
conservative.

\appendix

\section{Details of the Proof of Proposition \texorpdfstring{\ref{pro:probbaly-stable-scheme-1}}{}}
\label{app:details}

The change of the norm for general $L$ can be calculated explicitly as
\begin{align*}
\stepcounter{equation}\tag{\ref{eq:probably-stable-scheme-proof1}}
  &
  \norm{u_+}^2 - \norm{u_0}^2
  =
  \dt \biggl(
      \scps{ u_0 }{ L_0 u_0 }
    + \scps{ u_0 }{ L_1 u_0 }
  \biggr)
  \\
  &\quad
  + \frac{1}{2} \dt^2 \biggl(
      \scps{ u_0 }{ L_0^2 u_0 }
    - \scps{ u_0 }{ L_0 L_1 u_0 }
    + \scps{ u_0 }{ L_1 L_0 u_0 }
    + \scps{ u_0 }{ L_1^2 u_0 }
    \\
    &\qquad\quad
    + \frac{1}{2} \norm{ L_0 u_0 }^2
    + \scps{ L_0 u_0 }{ L_1 u_0}
    + \frac{1}{2} \norm{ L_1 u_0 }^2
  \biggr)
  \\
  &\quad
  + \frac{1}{2} \dt^3 \biggl(
    - \scps{ u_0 }{ L_0 L_1 L_0 u_0 }
    + \scps{ u_0 }{ L_1 L_0^2 u_0 }
    - \scps{ u_0 }{ L_1 L_0 L_1 u_0 }
    + \scps{ u_0 }{ L_1^2 L_0 u_0 }
    \\
    &\qquad\quad
    + \frac{1}{2} \scps{ L_0 u_0 }{ L_0^2 u_0 }
    - \frac{1}{2} \scps{ L_0 u_0 }{ L_0 L_1 u_0 }
    + \frac{1}{2} \scps{ L_0 u_0 }{ L_1 L_0 u_0 }
    + \frac{1}{2} \scps{ L_0 u_0 }{ L_1^2 u_0 }
    \\
    &\qquad\quad
    + \frac{1}{2} \scps{ L_0^2 u_0 }{ L_1 u_0 }
    + \frac{1}{2} \scps{ L_1 u_0 }{ L_1 L_0 u_0 }
    - \frac{1}{2} \scps{ L_1 u_0 }{ L_0 L_1 u_0 }
    + \frac{1}{2} \scps{ L_1 u_0 }{ L_1^2 u_0 }
  \biggr)
  \\
  &\quad
  + \frac{1}{2} \dt^4 \biggl(
    - \scps{ u_0 }{ L_1 L_0 L_1 L_0 u_0 }
    - \frac{1}{2} \scps{ L_0 u_0 }{ L_0 L_1 L_0 u_0 }
    + \frac{1}{2} \scps{ L_0 u_0 }{ L_1 L_0^2 u_0 }
    \\
    &\qquad\quad
    - \frac{1}{2} \scps{ L_0 u_0 }{ L_1 L_0 L_1 u_0 }
    + \frac{1}{2} \scps{ L_0 u_0 }{ L_1^2 L_0 u_0 }
    + \frac{1}{8} \norm{L_0^2 u_0 }^2
    - \frac{1}{4} \scps{ L_0^2 u_0 }{ L_0 L_1 u_0 }
    \\
    &\qquad\quad
    + \frac{1}{4} \scps{ L_0^2 u_0 }{ L_1 L_0 u_0 }
    + \frac{1}{4} \scps{ L_0^2 u_0 }{ L_1^2 u_0 }
    + \frac{1}{2} \scps{ L_1 u_0 }{ L_1 L_0^2 u_0 }
    \\
    &\qquad\quad
    - \frac{1}{4} \scps{ L_0 L_1 u_0 }{ L_1 L_0 u_0 }
    + \frac{1}{8} \norm{ L_1 L_0 u_0 }^2
    + \frac{1}{4} \scps{ L_1 L_0 u_0 }{ L_1^2 u_0 }
    \\
    &\qquad\quad
    - \frac{1}{2} \scps{ L_1 u_0 }{ L_0 L_1 L_0 u_0 }
    + \frac{1}{2} \scps{ L_1 u_0 }{ L_1^2 L_0 u_0 }
    - \frac{1}{2} \scps{ L_1 u_0 }{ L_1 L_0 L_1 u_0 }
    \\
    &\qquad\quad
    + \frac{1}{8} \norm{ L_0 L_1 u_0 }^2
    - \frac{1}{4} \scps{ L_0 L_1 u_0 }{ L_1^2 u_0 }
    + \frac{1}{8} \norm{ L_1^2 u_0 }^2
  \biggr)
  + \O( \dt^5 ).
\end{align*}
This calculation is also provided as Mathematica notebook in the
accompagnying repository \cite{ranocha2019energyRepro}.
Using the skew-symmetry of the operators $L_0, L_1$, the term proportional
to $\dt$ vanishes.
Since $L_0$ and $L_1$ commute, the $\dt^2$ term can be simplified to
\begin{align*}
\stepcounter{equation}\tag{\theequation}
  &
  + \frac{1}{2} \dt^2 \biggl(
      \scps{ u_0 }{ L_0^2 u_0 }
    + \scps{ u_0 }{ L_1^2 u_0 }
    + \frac{1}{2} \norm{ L_0 u_0 }^2
    + \scps{ L_0 u_0 }{ L_1 u_0}
    + \frac{1}{2} \norm{ L_1 u_0 }^2
  \biggr)
  \\
  =&
  + \frac{1}{2} \dt^2 \biggl(
    - \norm{ L_0 u_0 }^2
    - \norm{ L_1 u_0 }^2
    + \frac{1}{2} \norm{ L_0 u_0 }^2
    + \scps{ L_0 u_0 }{ L_1 u_0}
    + \frac{1}{2} \norm{ L_1 u_0 }^2
  \biggr)
  \\
  =&
  - \frac{1}{4} \dt^2 \norm{L_0 u_0 - L_1 u_0}^2.
\end{align*}
Similarly, using $L_0 L_1 = L_1 L_0$ and the skew-symmetry of both operators,
the $\dt^3$ term can be shown to vanish. For example,
\begin{equation}
  - \scps{ u_0 }{ L_0 L_1 L_0 u_0 } + \scps{ u_0 }{ L_1 L_0^2 u_0 }
  =
  + \scps{ L_0 u_0 }{ L_1 L_0 u_0 } - \scps{ L_0 u_0 }{ L_1 L_0 u_0 }
  =
  0.
\end{equation}
Using similar simplifications, the $\dt^4$ term can also be simplified.
Firstly, the terms with an odd number of $L_0$ or $L_1$ cancel each other,
leaving
\begin{multline}
  - \scps{ u_0 }{ L_1 L_0 L_1 L_0 u_0 }
  + \frac{1}{8} \norm{L_0^2 u_0 }^2
  + \frac{1}{4} \scps{ L_0^2 u_0 }{ L_1^2 u_0 }
  + \frac{1}{2} \scps{ L_1 u_0 }{ L_1 L_0^2 u_0 }
  + \frac{1}{8} \norm{ L_1^2 u_0 }^2
  \\
  - \frac{1}{4} \scps{ L_0 L_1 u_0 }{ L_1 L_0 u_0 }
  + \frac{1}{8} \norm{ L_1 L_0 u_0 }^2
  - \frac{1}{2} \scps{ L_1 u_0 }{ L_0 L_1 L_0 u_0 }
  + \frac{1}{8} \norm{ L_0 L_1 u_0 }^2.
\end{multline}
Rewriting the remaining terms yields
\begin{multline}
  - \scps{ L_0^2 u_0 }{ L_1^2 u_0 }
  + \frac{1}{8} \norm{L_0^2 u_0 }^2
  + \frac{1}{4} \scps{ L_0^2 u_0 }{ L_1^2 u_0 }
  - \frac{1}{2} \scps{ L_0^2 u_0 }{ L_1^2 u_0 }
  + \frac{1}{8} \norm{ L_1^2 u_0 }^2
  \\
  - \frac{1}{4} \scps{ L_0^2 u_0 }{ L_1^2 u_0 }
  + \frac{1}{8} \scps{ L_0^2 u_0 }{ L_1^2 u_0 }
  + \frac{1}{2} \scps{ L_0^2 u_0 }{ L_1^2 u_0 }
  + \frac{1}{8} \scps{ L_0^2 u_0 }{ L_1^2 u_0 }.
\end{multline}
Finally, these terms can be combined to
\begin{equation}
  + \frac{1}{8} \norm{L_0^2 u_0 }^2
  - \frac{3}{4} \scps{ L_0^2 u_0 }{ L_1^2 u_0 }
  + \frac{1}{8} \norm{ L_1^2 u_0 }^2,
\end{equation}
resulting in \eqref{eq:probably-stable-scheme-proof2}.

\appendix

\section*{Acknowledgements}

Research reported in this publication was supported by the King Abdullah
University of Science and Technology (KAUST).
The first author was partially supported by the German Research Foundation
(DFG, Deutsche Forschungsgemeinschaft) under Grant SO~363/14-1.

\printbibliography

\end{document}